%% file: Multiplicity_two_min-max_v9.tex
\documentclass[11pt]{amsart}
\usepackage{amssymb}
\usepackage{amsfonts}
\usepackage{amsmath}
\usepackage{graphicx}
\usepackage{xcolor}
\usepackage{amsmath}
\usepackage{mathrsfs}
\usepackage{stmaryrd}
\usepackage{epsfig,color}
\usepackage{blindtext}
\usepackage{enumerate}
\usepackage{hyperref}
\usepackage{url}
\usepackage{bbm}
\usepackage{filecontents}
\usepackage{nicefrac,mathtools}
\usepackage{bm}   
\DeclareGraphicsExtensions{.pdf,.jpeg,.png}
\usepackage{epstopdf}
\usepackage{cancel} 
\usepackage[normalem]{ulem} 
\usepackage{verbatim} 

\usepackage{color}
\usepackage[msc-links, lite]{amsrefs}
\usepackage{geometry}
\newtheorem{theorem}{Theorem}[section]

\newtheorem{proposition}[theorem]{Proposition}
\newtheorem{lemma}[theorem]{Lemma}
\newtheorem{corollary}[theorem]{Corollary}

\theoremstyle{definition}

\theoremstyle{remark}
\newtheorem{remark}[theorem]{Remark}
\newtheorem{example}[]{Example}

\numberwithin{equation}{section}

\newcommand{\dv}{\mathrm{div}}

\newcommand{\mf}{\mathbf}
\newcommand{\mb}{\mathbb}
\newcommand{\mc}{\mathcal}

\newcommand{\mk}{\mathfrak}

\newcommand{\wti}{\widetilde}

\newcommand{\Area}{\mathrm{Area}}

\newcommand{\dist}{\operatorname{dist}}

\newcommand*{\avint}{\mathop{\ooalign{$\int$\cr$-$}}}

\newcommand{\n}{\mathbf n}

\DeclareMathOperator{\Index}{index}

\DeclareMathOperator{\Ric}{Ric}

\title{Min-max minimal hypersurfaces with higher multiplicity}

\date{\today}

\author{Zhichao Wang}
\address{Department of Mathematics, University of British Columbia, Vancouver, BC V6T 1Z2, Canada}
\email{zhichao@math.ubc.ca}

\author{Xin Zhou}
\address{Department of Mathematics, 531 Malott Hall, Cornell University,	Ithaca, NY 14853, and Department of Mathematics, University of California Santa Barbara, Santa Barbara, CA 93106, USA}
\email{xinzhou@cornell.edu}

\begin{document}

\begin{abstract}
We exhibit the first set of examples of non-bumpy metrics on the $(n+1)$-sphere ($2\leq n\leq 6$) in which the varifold associated with the two-parameter min-max construction must be a multiplicity-two minimal $n$-sphere. This is proved by a new area-and-separation estimate for certain minimal hypersurfaces with Morse index two inspired by an early work of Colding-Minicozzi. We also construct non-bumpy projective spaces in which the first min-max hypersurfaces are one-sided, and non-bumpy balls in which the free boundary min-max hypersurfaces are improper.
\end{abstract}

\maketitle

\section{Introduction}

In the past decade, we have witnessed many important advancement in the development for minimal hypersurfaces, including the solution of Yau conjecture on minimal surfaces by Marques-Neves \cite{MN17}, and Song \cite{Song18}, and the establishment of a Morse theory for the area functional by Marques-Neves \citelist{\cite{MN16}\cite{MN18}}. One key challenge in these works is the a priori existence of integer multiplicity of the varifolds produced by the Almgren-Pitts min-max theory \citelist{\cite{Alm62}\cite{Alm65}\cite{Pi}} (see also \cite{Colding-DeLellis03, DeTa}). Now we have very well understanding of the multiplicity when the ambient manifold has a bumpy metric (\cite{Whi91}) thanks to the solution of the Multiplicity One Conjecture \cite{Zhou19}; (see also \cite{CM20}). For non-bumpy metrics, there are known trivial examples where some min-max varifolds have higher multiplicities, while some others have multiplicity one. For instance, on a thin and long flat torus, the min-max varifold can be two identical copies of the cross section, but one can move one copy parallelly away to obtain a multiplicity one varifold of the same mass. This made it tempting to conjecture that for any metric there always exists a min-max varifold of multiplicity one. However, in this paper, we will disprove this conjecture by constructing the first set of nontrivial and non-bumpy examples, where the varifold associated with the two-parameter min-max construction must have multiplicity two.

 \begin{theorem}\label{thm:main}
 	The $(n+1)$-sphere $S^{n+1}$ of dimension $3\leq (n+1)\leq 7$ admits metrics $g$ with non-negative Ricci curvature so that the second volume spectrum $\omega_2(S^{n+1},g)$ can only be achieved by a degenerate minimal $n$-sphere with multiplicity 2.	
 \end{theorem}
Here $\{\omega_k\}$ is the {\em volume spectrum of $M$} defined by Gromov \cite{Gro88}, Guth \cite{Guth09}, and Marques-Neves \cite{LMN16} as a sequence of non-decreasing positive numbers 
\[ 0<\omega_1(M, g)\leq\omega_2(M, g)\leq\cdots \leq\omega_k(M, g) \to \infty, \]
depending only on $M$ and $g$; see also 
\cite{Lio16}\cite{GL17}\cite{Sab17}\cite{Wang21} for further studies. 
  
  \medskip


Our proof relies on a new area comparison argument. Roughly speaking, we show that if a connected multiplicity-one minimal hypersurface $\Sigma$ is sufficiently close to a multiplicity-two degenerate stable hypersurface $S^n_0$, then the area of $\Sigma$ is strictly greater than that of $S_0^n$. Similar area comparison arguments have played essential roles in the study of Morse index and multiplicity of min-max minimal hypersurfaces (see \cite{MN12, Zhou15, Zhou17, KMN16}), but in a reserve way. In particular, the catenoid estimates by Ketover-Marques-Neves \cite{KMN16} imply that if $S^n_0$ is unstable, then by pushing away the two copies of $S^n_0$ and then adding a catenoid type neck therein, one can strictly decrease the area. More precisely, the area-decrease by pushing away the two copies of $S^n_0$ dominates the area-increase by adding the catenoid neck, when $S^n_0$ is unstable. This idea has been further extended by Haslhofer-Ketover \cite{HK19} (see also \cite{Dea}) to construct the second minimal 2-sphere in $S^3$ with a bumpy metric associated with the second width. 
In our situation, we will show the reverse phenomenon when $S^n_0$ is degenerate stable.  A crucial ingredient is to bound the area difference between $\Sigma$ and $2S^n_0$ away from the neck region. 
To do so, we need a new distance-separation estimate between $\Sigma$ and $S^n_0$ in terms of the size of the neck region; see Theorem \ref{thm:hausdorff distance} and \ref{thm:hausdorff distance:high}. 
This part is inspired by an early work of Colding-Minicozzi \cite{CM02}; see the part on idea of proof for more details.

\medskip
The Multiplicity One Theorem also implies that for bumpy metrics, the min-max minimal hypersurfaces can only be two-sided. However, our method can also be used to construct non-bumpy metrics on the projective spaces in which the first width must be achieved by one-sided hypersurfaces. 
\begin{corollary}\label{cor:one-sided}
	The $(n+1)$-projective space $\mb{RP}^{n+1}$ with $3\leq (n+1)\leq 7$ admits metrics $g$ so that the first volume spectrum $\omega_1(\mb {RP}^{n+1},g)$ can only be achieved by a one-sided $\mb {RP}^n$ with multiplicity 2.
\end{corollary}

\medskip
In \cite{LZ16}, Li-Zhou developed a free boundary min-max theory for compact manifolds with boundary $(M, \partial M)$. The minimal hypersurface $\Sigma$ constructed therein may {\em not be properly embedded}, i.e. it may have interior touching with the boundary $\operatorname{int}(\Sigma)\cap \partial M \neq \emptyset$. The touching phenomenon had caused major challenges in the application of this theory, e.g. \cite{GLWZ19, Wang_21_2, SWZ}. It has been conjectured that the touching could happen even for Euclidean domains \cite[Conjecture 1.5]{LZ16}. Here we exhibit examples where the touching phenomenon does happen.


\begin{corollary}
$B^{n+1}$ admits a metric with minimal boundary and non-negative Ricci curvature so that its first volume spectrum can only be achieved by its boundary with multiplicity one.	
	\end{corollary}
\begin{proof}
Let $(S^{n+1},g)$ be as in Theorem \ref{thm:main}. Then the round $S_0^n$ divides the $(n+1)$-sphere into two connected components, denoted by $M_+$ and $M_-$; see Section \ref{sec:main proof} for explicit description. 
Note that $(M_+,g)$ is an $(n+1)$-ball that has positive Ricci curvature away from its boundary. By \cite{SWZ},  $\omega_1(M_+,g)$ is realized by a free boundary minimal hypersurface $\Sigma$ with multiplicity one, whose index is bounded by one from above. Then by reflection along $S_0^n$, one can obtain a smoothly embedded minimal hypersurface $\wti \Sigma\subset (S^{n+1},g)$ with index less than or equal to 2. Clearly, $\omega_1(M_+,g)\leq \Area(S_0^n)$. It follows that 
\[  \Area(\wti \Sigma)\leq 2\Area(S_0^n).   \]
Then by Theorem \ref{thm:main}, $\wti \Sigma$ can only be the $S_0^n$ with multiplicity two, which implies that $\Sigma$ is the boundary of $M_+$.
	\end{proof}

\subsection*{Idea of the proof}
We construct a sequence of Riemannian $(n+1)$-spheres $M_k$ (isometrically embedded in $\mb R^{n+2}$) that converges locally smoothly to $S_0^n\times \mb R$ (see Section \ref{sec:main proof} for details), where $S_0^{n}$ is a round $n$-sphere in $\mb R^{n+1}$ and embedded in $M_k$ as the unique degenerate stable minimal hypersurface for each $k$. Moreover, for each $k$, the level sets (denoted by $\{S_t\}$) of the distance function to $S_0^n$ have area lower bound $\Omega_n(1-|t|^{n+1})$, where $\Omega_n$ is the volume of the unit $n$-sphere.

Suppose on the contrary that each $\omega_2(M_k)$ is realized by a minimal hypersurface $\Sigma_k$ 
that is not $S_0^n$. Then by our construction, $\Sigma_k$ has to be connected, unstable, of index less than or equal to 2 (by \cite{MN16}) and multiplicity one (\cite{Zhou19}). Then the compactness \cite{Sharp17} gives that $\Sigma_k$ converges locally smoothly to an embedded minimal hypersurface with integer multiplicity in $S_0^n\times \mb R$. By the lower bounds for $\omega_2(M)$, the limit of $\Sigma_k$ is exactly $S_0^n$ with multiplicity 2. Moreover, by the classification of embedded minimal hypersurfaces with two ends (\cite{Sch83_1}), after suitable scaling, the blowup around each singular point is a standard catenoid. Let $r_k$ denote the ``radius'' of link of the small catenoid in $\Sigma_k$, i.e. the distance of the small catenoid to the center point. 
First, the Hausdorff distance between $\Sigma_k$ and $S_0^n$ is bounded from above by $10r_k|\log r_k|$ (to be discussed next). 
After replacing the annuli in neck regions of $\Sigma_k$ by two minimizing $n$-disks, by using the one-sided minimizing property of $\{S_t\}$, the area of the new hypersurface is at least $2|S^n_0| - \mc O(r^{n+1}_k|\log r_k|^{n+1})$; 
see \eqref{eq:bound new surface}. On the other hand, the neck regions of $\Sigma_k$ contribute at least $c(n)r_k^n$ amount of area more than that of the $n$-disks; see \eqref{eq:scaling difference}.  
Therefore, $|\Sigma_k| - 2 |S^n_0| \geq c(n)r_k^n - \mc O(r^{n+1}_k|\log r_k|^{n+1})>0$, 
contradicting $\omega_2(M_k)\leq 2\omega_1(M_1)\leq 2\Area(S_0^n)$.

To bound the Hausdorff distance $d_H(\Sigma_k, S_0^n)$, the key challenge arises from intermediate annuli regions $A(y_k, R_k, \epsilon;M_k)$, where $y_k$ is a singular point, $R_k/r_k\nearrow \infty$, $R_k\searrow 0$. In fact, $\Sigma_k$ is close in smooth topology to the catenoid of radius $r_k$ inside $B(y_k, R_k;M_k)$ and to $2 S^n_0$ outside $B(y_k, \epsilon;M_k)$ by smooth convergence. However, $\Sigma_k$ may not be a 2-sheeted graph over $S^n_0$ inside $A(y_k, R_k, \epsilon;M_k)$, and we are forced to write one component $\Sigma_k^2$ (of $\Sigma_k$) as a graph over the other $\Sigma_k^1$. First, $d_H(\Sigma_k^1, \Sigma_k^2)$ when restricted to $\partial B(y_k, R_k;M_k)$ is at most $3r_k\log\frac{R_k}{r_k}$; see \eqref{eq:Ik Rk}. We need to prove that this bound keeps roughly at the order $r_k\log\frac{s}{r_k}$ on $\partial B(y_k, s;M_k)$ when $s\nearrow \epsilon$. Our proof relies on several new monotonicity formulas inspired by Colding-Minicozzi \cite{CM02}. In particular, we study carefully the evolution of the averages of the height function $w_k$ (between $\Sigma_k^1, \Sigma_k^2$): $s\mapsto\avint_{\Sigma_k^1\cap\partial B(y_k, s;M_k)}w_k$; see Proposition \ref{prop:bound wk} and compare with \cite[Lemma 2.1]{CM02}. The main new challenge as compared with \cite{CM02} is that $w_k$ is not a graph function over a fixed hypersurface, and this causes many higher-order terms in our monotonicity formulas, particularly as the height function $w_k$ only satisfies a highly nonlinear PD-inequality \eqref{eq:upper bound for laplace wk}. Note that some extra care is needed when the two singular points are not too far from each other; see Proposition \ref{prop:bound wk both}.

\subsection*{Outline}
In Section \ref{sec:blowup}, we will classify the limit cones of the singular set arising from the compactness of minimal hypersurfaces. In Section \ref{sec:main proof}, we describe the concrete constructions and prove our main results by assuming the key Hausdorff distance upper bound estimates, which 
will be proved in Section \ref{sec:hasdorff distance}. Finally in Appendix \ref{sec:graph functions}, we derive a general inequality for minimal graphs over another minimal hypersurface. Then we list some basic results of catenoids in Appendix \ref{sec:catenoids}.

\subsection*{Acknowledgments}
We would like to thank Professor Gang Tian for his constant support and encouragement. X.Z. was supported by NSF grant DMS-1945178,  and an Alfred P. Sloan Research Fellowship.

\section{Blowing-up analysis}\label{sec:blowup}

Let $\{M_k\}$ 
be a sequence of $(n+1)$-spheres embedded in $\mb R^{n+2}=\mb R^{n+1}\times \mb R$. Denote by $S_0^n$ the unit sphere in $\mb R^{n+1}$. Let $\Sigma_k$ be a closed embedded connected minimal hypersurface in $M_k$. We use $B_r(p)$ and $B(p,r; M_k)$ to denote the geodesic ball in $\mb R^{n+1}$ and $M_k$, respectively. In this section, $M_k$ and $\Sigma_k$ always satisfy the following requirements.
\begin{enumerate}
	\item For each $k$, $S_0^n$ is a stable minimal hypersurface in $M_k$ with constant Jacobi fields;
	\item $M_k$ locally smoothly converges to $S_0^n\times \mb R$ in $\mb R^{n+2}$ as $k\to\infty$;
	\item $\Sigma_k$ converges to twice of $S_0^n$ in the sense of varifolds;
	\item $\Sigma_k$ has Morse index less than or equal to 2.
	\end{enumerate}

\begin{example}	\label{example}
Let $M_k$ be the $(n+1)$-sphere given by 
\[  x_1^2+x_2^2+\cdots+x_{n+1}^2+\frac{x_{n+2}^{2n}}{k^{2n}}=1 . \]
Suppose that $M_k$ contains an embedded minimal hypersurface $\Sigma_k$ with index less than or equal to 2 and 
\[ \Area(S^n_0)+\delta\leq \Area(\Sigma_k)\leq 3 \Area(S^n_0)-\delta \]
for some $\delta >0$. Then $\{\Sigma_k\}$ and $\{M_k\}$ satisfy our requirements in this section. 	
	\end{example}

By compactness \cite{Sharp17}, $\Sigma_k$ locally smoothly converges to $S_0^n$ with multiplicity two away from a set $\mc W$ consisting of one or two points. 
Let $p\in\mc W$. Then we can take a positive constant $\epsilon<10^{-1000}$ small enough such that the convergence for $\Sigma_k$ is locally smooth in $ B_{2\epsilon}(p)\setminus \{p\}$, and 
 \begin{equation}\label{eq:density upper bound}
  \Area(B(p,2\epsilon;M_k)\cap \Sigma_k)\leq \frac{5}{2} \cdot \frac{\Omega_{n-1}}{n}(2\epsilon)^n, 
  \end{equation}
  where $\Omega_m$ is the volume of unit $m$-spheres. Choose $p_k\in\Sigma_k$ so that 
\[|A^{\Sigma_k}(p_k)|=\max_{\Sigma_k\cap B(p,\epsilon;M_k)}|A^{\Sigma_k}(x)|.\] 
Clearly, $|A^{\Sigma_k}(p_k)|\to\infty$ as $k\to\infty$. In the following, we classify the limit cones of $\{\Sigma_k\}$ at $p$, which is either a hyperplane with multiplicity or a catenoid.
\begin{proposition}[Classification of limit cones]\label{prop:limit cones}
Let $\{c_k\}$ be a sequence of positive numbers with $c_k\to\infty$. 
\begin{enumerate}
	\item If $\lim_{k\to\infty}|A^{\Sigma_k}(p_k)|/c_k=\ell>0$, then $c_k(\Sigma_k-p_k)$ locally smoothly converges to a catenoid contained in $\mb R^{n+1}$;
	\item If $\lim_{k\to\infty}|A^{\Sigma_k}(p_k)|/c_k=0$, then $c_k(\Sigma_k-p_k)$ locally smoothly converges to a hyperplane contained in $\mb R^{n+1}$;
	\item If $\lim_{k\to\infty}|A^{\Sigma_k}(p_k)|/c_k=\infty$, then $c_k(\Sigma_k-p_k)$  converges to a multiplicity-two hyperplane contained in $\mb R^{n+1}$. Moreover, the convergence is locally smooth away from at most two points including 0.
	\end{enumerate}
\end{proposition}
\begin{proof}
Note that $c_k(\Sigma_k-p_k)$ is always a minimal hypersurface in $c_k(M_k-p_k)$ with $\Index\leq 2$. By the monotonicity formula and \eqref{eq:density upper bound}, 
\begin{equation}
\label{eq:density bounds}
 \lim_{k\to\infty} \Area(B_r(0)\cap c_k(\Sigma_k-p_k))\leq \frac{5}{2}\cdot \frac{\Omega_{n-1}}{n} r^n .
 \end{equation}
Clearly, $c_k(M_k-p_k)$ converges locally smoothly to $\mb R^{n+1}$. Thus $c_k(\Sigma_k-p_k)$ converges to an embedded minimal hypersurface with integer multiplicity in $\mb R^{n+1}$ with $\Index\leq 2$. By the work of Tysk \cite{Tys}, the limit minimal hypersurface, which has finite Morse index and polynomial volume growth~\eqref{eq:density bounds}, must have a unique tangent cone at infinity; therefore using the work of Schoen  \cite{Sch83_1}*{Theorem 3} (see also \cite{HHW}*{Theorem 1.3}), the limit minimal hypersurface is either a catenoid or a hyperplane.

If $\lim_{k\to\infty}|A^{\Sigma_k}(p_k)|/c_k=\ell<\infty$, then for any compact set $\Omega\subset \mb R^{n+2}$, $c_k(\Sigma_k-p_k)\cap \Omega$ has uniformly bounded second fundamental form, and this implies that the convergence is locally smooth. When $\ell>0$, the limit hypersurface has at least one point that has non-zero curvature. Thus the limit is a catenoid. When $\ell=0$, the limit hypersurface is flat and hence is a hyperplane.

If $\lim_{k\to\infty}|A^{\Sigma_k}(p_k)|/c_k=\infty$, then the convergence is non-smooth. By Allard's Regularity Theorem \cite{Si}*{Theorem 24.2}, the multiplicity of the convergence is larger than or equal to two. Thus the limit is a stable minimal hypersurface in $\mb R^{n+1}$ with polynomial (degree $n$) area growth, which can only be a hyperplane by Schoen-Simon \cite{SS} and Schoen-Simon-Yau \cite{SSY}. Since the index of $\Sigma_k$ is bounded from above by 2, the convergence is locally smooth away from at most two points. Together with \eqref{eq:density bounds}, we also have that the multiplicity of the convergence is exactly 2. Hence Proposition \ref{prop:limit cones} is proved. 
	\end{proof}

\begin{remark}\label{rmk:non-smooth convergence}
Let $x_k\in \Sigma_k$ and $c_k\to \infty$. Then $c_k(\Sigma_k-x_k)$	converges to a multiplicity-one catenoid or hyperplane with multiplicity one or two in the sense of varifold. Suppose that the convergence is not locally smooth. Then the limit is a multiplicity-two hyperplane.
	\end{remark}

Denote by $r_{k,1}=\sqrt {n(n-1)}|A^{\Sigma_k}(p_k)|^{-1}$. Take $y_{k,1}\in \Sigma_k$ so that $r_{k,1}^{-1}(\Sigma_k-y_{k,1})$ converges to a standard catenoid $\mc C$ in $\mb R^{n+1}$; see Appendix \ref{sec:catenoids} for several properties of the geometry of $\mc C$. In the remaining of this section, we are going to find a sequence of ``bad balls'', in which $\Sigma_k$ is not flat enough even after rescaling.

So long as
\[ \max_{x\in B(y_{k,1},\epsilon;M_k)\cap \Sigma_k}|A^{\Sigma_k}(x)||x-y_{k,1}|<2\sqrt{n(n-1)},\]
we let $y_{k,2}:=y_{k,1}$. Otherwise, take $y'_{k,2}\in \Sigma_k$ so that 
\[ \max_{x\in B(y_k,\epsilon;M_k)\cap \Sigma_k}|A^{\Sigma_k}(x)|x-y_k|=|A^{\Sigma_k}(y_{k,2}')||y_{k,2}'-y_{k,1}|\geq 2\sqrt{n(n-1)} .\]
Then by Proposition \ref{prop:limit cones}, $|A^{\Sigma_k}(y_{k,2}')|(\Sigma_k-y_{k,2}')$ converges locally smoothly to a catenoid in $\mb R^{n+1}$. Then we can take $y_{k,2}$ around $y_{k,2}'$ and $r_{k,2}>0$ so that $r_{k,2}^{-1}(\Sigma_k-y_{k,2})$ converges locally smoothly to a standard catenoid. Let 
\[ r_k=\max\{r_{k,1},r_{k,2}\}. \]
Denote by 
\begin{equation}
\label{eq:b_k}
b_k:=|y_{k,1}-y_{k,2}|.
\end{equation}  
By the locally smooth convergence of $\Sigma_k$ in $B_{\epsilon}(p)\setminus\{p\}$, we know that \[ b_k\to 0, \text{ as } k \to\infty .\] 
Note that in a standard catenoid (see Appendix \ref{sec:catenoids}), we have $|x|\cdot |A(x)|\leq\sqrt{n(n-1)}$. 
Therefore, we also have that 
\[ b_k/r_k\to\infty, \text{ if } y_{k,1}\neq y_{k,2}.\]
For simplicity, denote by $A(p,r,s;M_k)=B(p,s;M_k)\setminus B(p,r;M_k)$, $\mc W_k=\{y_{k,1}, y_{k,2}\}$ and 
\begin{equation}
\label{eq:d_k}
d_k(x)=\dist_{M_k}(x,\mc W_k).
\end{equation}

We now claim that for all sufficiently large $k$,
\begin{equation}\label{eq:curvature}
  \max_{x\in B(p,\epsilon;M_k)\cap \Sigma_k}|A^{\Sigma_k}(x)|d_k(x)<2\sqrt{n(n-1)}.  
  \end{equation}
Suppose not, then one can find $y_{k,3}$ and another small cantenoid. Observe that each small catenoid will contribute index 1. This contradicts that $\Sigma_k$ has index less than or equal to 2.

From now on, let $\Sigma_k^1$ and $\Sigma_k^2$ be the two connected components of $\Sigma_k\setminus \cup_{j}B(y_{k,j},2r_{k,j};M_k)$. Without loss of generality,  outside $B(p,\epsilon;M_k)$, we choose the unit normal vector field of $\Sigma_k^1$ pointing towards $\Sigma_k^2$.
\begin{proposition}
\label{prop:2nd order estimates}
	Given $\delta>0$, there exists $K>0$ such that for all $k\geq K$, the following statements hold true.
	\begin{enumerate}
	\item\label{item:A bound} 	The second fundamental forms satisfy for $x\in \Sigma_k\setminus \cup_{j=1}^2B(y_{k,j},Kr_{k,j};M_k)$,
  \begin{equation}\label{eq:|A| small}
 d_k\cdot|A^{\Sigma_k}(x)| +d^2_k\cdot|\nabla A^{\Sigma_k}(x)|<\delta. 
  \end{equation}
   Moreover, $\Sigma_k^1$ is a minimal graph over $\Sigma_k^2$.
   Denote by $w_k$ the graph function.
   \item\label{item:der bound} For every $x\in \Sigma_k^1\setminus \cup_{j=1}^2B(y_{k,j},Kr_{k,j};M_k)$,
    \begin{equation}\label{eq:derivative of graphs}
   d_k^2\frac{|\nabla^2 w_k|}{w_k}(x)+d_k\frac{|\nabla w_k|}{w_k}+\frac{w_k}{d_k} <\delta.
   \end{equation} 
   \item\label{item:harnack} For any  $s>K r_k$ and $x, y\in\Sigma_k^1\cap A(y_k,  s, 4s; M_k)\setminus B(z_k,s/4;M_k)$,
\[ 
  w_k(x)\leq (1+\delta)w_k(y).
  \]
\end{enumerate}
\end{proposition}

\begin{proof}
Note that by Remark \ref{rmk:non-smooth convergence}, the limit of $d^{-1}_k(x_k)(\Sigma_k-x_k)$ can only be a multiplicity-two hyperplane if $d_k(x_k)/r_{k,j}\to\infty$. Then the first item follows from standard blowup and contradiction arguments. Observe the third one follows directly from the second one. Hence it suffices to prove the second item.

Suppose on the contrary that there exists $\delta>0$ and $\alpha_k\to\infty$ such that for a subsequence of $k\geq \alpha_k$, there exists $x_k\in \Sigma_k\setminus \cup_{j=1}^2B(y_{k,j},\alpha_kr_{k,j};M_k)$ violating \eqref{eq:derivative of graphs}. So long as $\lim_{k\to\infty} d_k(x_k)\neq 0$, after renormalizations, the sequence, still denoted by $\{w_k\}$, will converge locally smoothly to a Jacobi field of $S_0^n$ away from at most two points (see \cite{Si87}), which is a constant function. This will give the desired contradiction.

It remains to consider the situation $d_k\to0$ and $d_k/r_{k,j}\geq \alpha_k\to\infty$. By Remark \ref{rmk:non-smooth convergence}, $d_k^{-1}(\Sigma_k-x_k)$ converges to a hyperplane with multiplicity two. Then the normalizations of $\{w_k\}$ will converge locally smoothly to a positive Jacobi field of the hyperplane, which is also a constant function.  This will also give a contradiction.
\end{proof}

\begin{figure}[h]
	\begin{center}
		\def\svgwidth{0.7\columnwidth}
		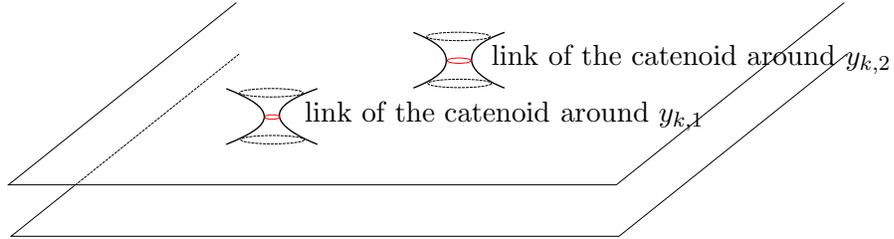
		\caption{Structure of $\Sigma_k$.}
		\label{fig:Sigmak structure}
	\end{center}
\end{figure}

\begin{remark}\label{rmk:structure of Sigmak}
For $\{\Sigma_k\}$ satisfying the requirements at the beginning of this section, we summarize the properties of $\Sigma_k$ (see Figure \ref{fig:Sigmak structure}) for large $k$:
\begin{enumerate}
	\item there exists $\{y_{k,j}\}$ containing one or two points such that for each $j$, $\Sigma_k$ has a small catenoid with radius $r_{k,j}$ around $y_{k,j}$;
	\item $\Sigma_k$ has exactly two connected components (denoted as $\Sigma_k^1$ and $\Sigma_k^2$) after removing such two (possibly one) small catenoids;
	\item $\Sigma_k^2$ is a graph over $\Sigma_k^1$ and the graph function satisfies Proposition \ref{prop:2nd order estimates}.
\end{enumerate}
\end{remark}

\section{Proof of the main theorem}\label{sec:main proof}
For any $a\geq 1$ and $3\leq (n+1)\leq 7$, let $M_a$ be the embedded $(n+1)$-sphere in $\mb R^{n+2}$ given by 
\[ x_1^2+ x_2^2\cdots+x_{n+1}^2+\frac{x_{n+2}^{2n}}{a^{2n}}=1.  \]
For simplicity, let $S_t:=\{x_{n+2}=t\}\cap M_a$, where we omit the $a$ without ambiguity. These $(n+1)$-spheres satisfy the following properties:
\begin{enumerate}[(A)]
	\item\label{item:Ric geq 0} it has non-negative Ricci curvature;
	\item\label{item:others unstable} $S^n_0=S_0$ is the unique stable minimal hypersurface for each $a\geq 1$; it follows that each minimal hypersurface is two-sided;
	\item $M_a\to S_0^n\times \mb R$ locally smoothly as $a\to\infty$;
	\item \label{item:cmc foliation}  $\{S_t\}_{t=-a}^a$ forms a foliation of $M_a$ and $\{x_{n+2}=t\}$ is an embedded $n$-sphere with mean curvature vector pointing away from  $S_0$ for $0<|t|<a$. In particular, any two embedded minimal hypersurfaces in $M_a$ intersect each other;
	\item the area of $S_t\subset M_a$ satisfies that for all $t$,
	\begin{equation}\label{eq:bound St}
	\Area(S_t)=\Omega_{n} (1-\frac{t^{2n}}{a^{2n}})\geq \Omega_n(1-|t|^{2n}) .
	\end{equation}
	\item for each $t$ with $0<|t|<a$,
	\[  |t|\leq \dist_{M_a}(S_0,S_t)\leq 2|t| . \] 
\end{enumerate}

Before stating our main results, we introduce the key height estimates, which will be proved in the next section (Theorem \ref{thm:hausdorff distance} and \ref{thm:hausdorff distance:high}).
\begin{theorem}\label{thm:Hausdorff distance0}
Let $\Sigma_k\subset M^{n+1}_k$ be a sequence of embedded minimal hypersurfaces with $\Index\leq 2$ and 
\[  \Area(\Sigma_k)\leq 3\Area(S_0^n)-\delta  \]	
for some $\delta>0$. Then for all sufficiently large $k$,
\[  \max_{\Sigma_k}\dist_{M_k}(x,S^n_0)\leq 8r_k|\log r_k|.\]
\end{theorem}

\begin{remark}
Note that the estimates are sharp for dimension $n+1=3$ (Theorem \ref{thm:hausdorff distance}). In higher dimensions $4\leq n+1\leq 7$, the height estimates are much better (Theorem \ref{thm:hausdorff distance:high}).
\end{remark}

In the following, we prove a rigidity theorem for minimal hypersurfaces with low index and area. This will be useful to identify the min-max solutions that realize the second width.
\begin{theorem}\label{thm:rigidity}
	Suppose that $\Sigma_k$ is a closed embedded minimal hypersurface in $M_k$ with $\Index \leq 2$ and 
	\[  \Area(\Sigma_k)\leq 2\Area(S_0^n).\]
	Then for sufficiently large $k$, $\Sigma_k$ is exactly $S^n_0$.
\end{theorem}
\begin{proof}
	Suppose not, then $\Sigma_k$ converges to a closed minimal hypersurface (with multiplicity) in $S^n_0\times \mb R$. By the monotonicity formula, the limit is $S^n_0\times\{0\}$ with multiplicity less than or equal to 2. We claim the multiplicity has to be two. So long as the the convergence has multiplicity one, the normalization of the difference between $S^n_0$ and $\Sigma_k$ converges to a non-trivial Jacobi field on $S^n_0$ (see \cite{Si87}). Note that such a Jacobi field is a constant. Thus $\Sigma_k$ lies on one side of $S^n_0$, which contradicts \eqref{item:cmc foliation} at the beginning of this section. Therefore, $\Sigma_k$ converges to $\Sigma$ with multiplicity two.
	
By \cite{Sharp17}, the convergence is smooth away from a set $\mc W$ containing at most two points since the index of $\Sigma_k$ is less than or equal to 2. Then $\{\Sigma_k\}$ and $\{M_k\}$ satisfy the conditions in Section \ref{sec:blowup} (see Example \ref{example}). Let $p\in\mc W$. By the argument therein, we can find $y_k(\in M_k)\to p$ and $r_k>0$ such that $r_k^{-1}(\Sigma_k-y_k)$ locally smoothly converges to a standard catenoid $\mc C\subset \mb R^{n+1}$, i.e. it has the center at 0 and 
	\[ \max_{x\in \mc C}|A(x)|=\sqrt {n(n-1)}. \]
By Remark \ref{rmk:structure of Sigmak}, here we have three possibilities:
	\begin{enumerate}
		\item\label{case:two separate points} $\mc W=\{p,q\}$ consists of two different points. In this case, there exists $z_k\to q$ and $\wti r_k$ such that $\wti r_k^{-1}(\Sigma_k-z_k)$ locally smoothly converges to a standard catenoid. Without loss of generality, we assume that $r_k\geq\wti r_k$ for all large $k$. 
		\item\label{case:one point} $\mc W=\{p\}$ consists of only one point and $b_k$ described in Proposition \ref{prop:2nd order estimates} is equal to 0. 
		It follows that $\Sigma_k\setminus B(y_k,2r_k;M_k)$ has exactly two connected components. In this case, we let $\wti r_k=0$.
		\item\label{case:multi 2 point} $\mc W=\{p\}$ consists of only one point and $b_k$ described in Proposition \ref{prop:2nd order estimates} is not 0, i.e. there exists $z_k\to p$, $|z_k-y_k|/r_k\to\infty$ and $\wti r_k>0$ such that $\wti r_k^{-1}(\Sigma_k-z_k)$ locally smoothly converges to a standard catenoid. Without loss of generality, we assume that $r_k\geq \wti r_k$ for large $k$.
	\end{enumerate}
	
	Above all, outside two small balls with radii $r_k$ and $\wti r_k$, $\Sigma_k$ has exactly two connected components; see Remark \ref{rmk:structure of Sigmak}. Moreover, we can take $R_k\to 0$ with $R_k/r_k\to\infty$, $r_k^{-1}(\Sigma_k-y_k)\cap B_{R_k/r_k}(0)$ is arbitrarily smoothly close to a standard catenoid $\mc C\cap B_{R_r/r_k}(0)$; if $z_k$ exists, we can also take $\wti R_k\to 0$ with $\wti R_k/\wti r_k\to\infty$ so that $\wti r_k^{-1}(\Sigma_k-z_k)\cap B_{R_k/r_k}(0)$ is arbitrarily close to a standard catenoid $\mc C\cap B_{R_k/r_k}(0)$.  

	Denote by $\gamma_k^1$ and $\gamma_k^2$ the two components of $\Sigma_k\cap \partial B(y_k,R_k;M_k)$. Then each $\gamma_k^j$ ($j=1$ or 2) bounds an area minimizing $n$-disk $\mc D_{k}^j$ in $M_k$. By Proposition \eqref{eq:area difference},
	\begin{equation}\label{eq:scaling difference}
	\Area(B(y_k,R_k;M_k)\cap \Sigma_k)-\sum_j\Area(\mc D_{k}^j)\geq \frac{\mc A_n}{2}r_k^n>0.
	\end{equation}
Here $\mc A_n$ ($n\geq 3$) is defined by \eqref{def:An} and $\mc A_2$ can be any fixed real numbers because of \eqref{eq:area difference}. Similarly, in Case \eqref{case:two separate points} or \eqref{case:multi 2 point}, there exist $\wti\gamma_k^j\subset \Sigma_k$ ($j=1,2$) surrounding $z_k$ such that each $\wti\gamma_k^j$ bounds an area minimizing $n$-disk $\wti{\mc D}_k^j$ in $M_k$ with 
	\begin{equation}\label{eq:bound small disks}
	\Area(B(z_k,\wti R_k;M_k)\cap \Sigma_k)-\sum_j\Area(\wti{\mc D}_k^j)\geq \frac{\mc A_n}{2}\wti r_k^n>0.
	\end{equation}
	Then we cut off $\Sigma_k\cap B(y_k,R_k;M_k)$ and $\Sigma_k\cap B(z_k,\wti R_k;M_k)$, which are close to a standard catenoid after scaling. After that, we add the $n$-disks $\{\mc D_{k}^j\}$ and $\{\wti{\mc D}_k^j \}$ to the new hypersurface; see Figure \ref{fig:replace}. This yields a closed hypersurface with two connected components $\Gamma_1$ and $\Gamma_2$ which are both homologous to $S^n_0$ in $M_k$. 
	
	\begin{figure}[h]
		\begin{center}
			\def\svgwidth{0.7\columnwidth}
			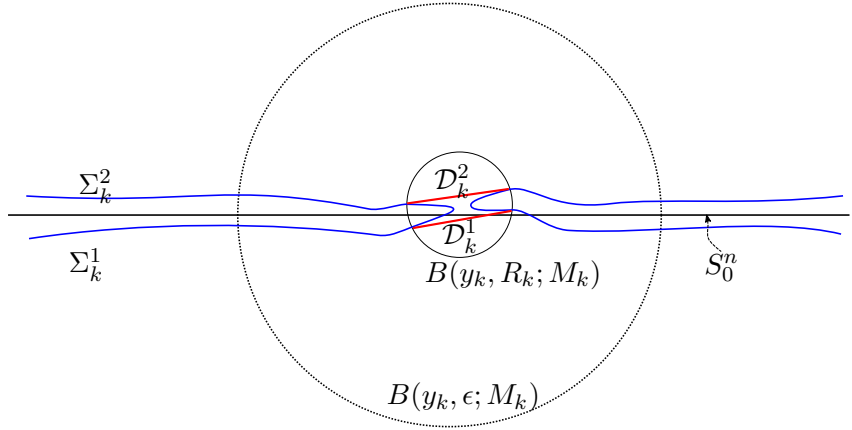
			\caption{Replacing the catenoids by two $n$-disks.}
			\label{fig:replace}
		\end{center}
	\end{figure}
	
        Note that Theorem \ref{thm:Hausdorff distance0} can be applied to obtain
	\[ \max_{x\in \Sigma_k}\dist_{M_k}(x,S_0^n)\leq 8r_k|\log r_k|; \]
	together with a standard minimal foliation argument (applied to $B(y_k, \epsilon; M_k)$), this implies that 
	\[   \max_{x\in \Gamma_j}\dist_{M_k}(x,S_0^n)\leq 10r_k|\log r_k|. \]
	Let $\mk d_k:=10r_k |\log r_k|$. Then by the one-sided minimizing property of $S_{\mk d_k}$ and $S_{-\mk d_k}$,
	\[  \Area(\Gamma_j)+\Area(S^n_0)>\Area(S_{\mk d_k})+\Area(S_{-\mk d_k}).   \]
	By \eqref{eq:bound St},
	\[  \Area(S_{\mk d_k})\geq \Area(S_0^n)-\Omega_n|\mk d_k|^{2n} \ \ \ \text{ and } \ \ \ \Area(S_{-\mk d_k})\geq \Area(S_0^n)-\Omega_n|\mk d_k|^{2n}; \]
	these imply that
	\begin{equation}\label{eq:bound new surface}
	\Area(\Gamma_1)+\Area(\Gamma_2)>2\Area(S_0^n)-4\Omega_n \mk d_k^{2n}. 
	\end{equation}
	By the construction of $\Gamma_j$,
	\begin{align*}
	\Area(\Sigma_k)&=\sum_{j=1}^2\Big(\Area(\Gamma_j)-\Area(\mc D_{k}^j)-\Area(\wti {\mc D}_k^j)\Big)+\Area(B(y_k,R_k;M_k)\cap \Sigma_k)\\
	& \ \ +\Area(B(z_k,\wti R_k;M_k)\cap \Sigma_k)\\
	&\geq \sum_{j=1}^2\Big(\Area(\Gamma_j)-\Area(\mc D_{k}^j)\Big)+\Area(B(y_k,R_k;M_k)\cap \Sigma_k)\\
	&> 2\Area(S_0^n)-10^{2n}r_k^{2n-1}\cdot (4\Omega_nr_k |\log r_k|^{2n})+\frac{\mc A_n}{2}r_k^n\\
	&>2\Area(S^n_0)+\frac{\mc A_n}{4}r_k^n.
	\end{align*}
	Here the first inequality is from \eqref{eq:bound small disks}; we used \eqref{eq:scaling difference} and \eqref{eq:bound new surface} in the second one; the last one follows from the fact that $r_k|\log r_k|^{2n}\to 0$ as $k\to \infty$. This gives a contradiction and we finish the proof of Theorem \ref{thm:rigidity}.
\end{proof}

Now we are going to prove our main results.
\begin{theorem}\label{thm:Mk}
	For sufficiently large $k$, the second width $\omega_2(M_k)$ can only be realized by $S_0^n$ with multiplicity two. 	
\end{theorem}
\begin{proof}
	By the work of Marques-Neves \cite{MN16}, together with the Frankel's property \eqref{item:cmc foliation}, 
	$\omega_2(M_k)$ is realized by the area of a connected, closed, embedded, minimal hypersurface $\Sigma_k$ with integer multiplicities and $\Index(\Sigma_k)\leq 2$. Observe that 
	\begin{equation}\label{eq:limit of omegak}
	 \lim_{k\to\infty}\omega_2(M_k)=\omega_2(S_0^n\times \mb R)=2\Area(S_0^n). 
	 \end{equation}
	
	We first prove that $\Sigma_k$ is $S_0^n$. Suppose not, then $\Sigma_k$ is unstable by Property \eqref{item:others unstable}; 
	therefore using the direct corollary of the Multiplicity One Theorem in \cite{Zhou19}, $\Sigma_k$ has multiplicity one, i.e. $\omega_2(M_k)=\Area(\Sigma_k)$. On the other hand, by the construction of optimal 1-sweepouts (c.f. \cite{Zhou15}), $\omega_1(M_k)=\Area(S_0^n)$. Then it follows that (c.f. \cite{MN17}*{Proof of Theorem 5.1}).
	\begin{equation}
	\omega_2(M_k)\leq 2\omega_1(M_k)=2\Area(S_0^n).
	\end{equation}
	By Theorem \ref{thm:rigidity}, $\Sigma_k$ has to be identical to $S_0^n$, but this contradicts \eqref{eq:limit of omegak}. Thus we conclude that $\Sigma_k$ has to be $S_0^n$. Then the multiplicity follows from \eqref{eq:limit of omegak}.

	Therefore, $\omega_2(M_k)$ can only be realized by $S_0^n$ with multiplicity two. 	
\end{proof}
 
Let $M_k$ be the Riemannian $(n+1)$-sphere in Theorem \ref{thm:Mk}. Denote by $N_k$ and $\mb{RP}^n_0$ the quotient spaces $M_k/\{x\sim -x\}$ and $S_0^n/\{x\sim-x\}$, respectively. Then we have the following properties:
\begin{enumerate}
	\item for each $k$, $N_k$ satisfies the Frankel's property, i.e. any two embedded minimal hypersurfaces intersect;
	\item $\mb{RP}^n_0$ is a one-sided embedded minimal hypersurface in $N_k$;
	\item as $k\to\infty$, $N_k$ locally smoothly converges to $S_0^n\times \mb R/\{x\sim-x\}$;
	\item $\omega_1(S_0^n\times \mb R/\{x\sim -x\})=\Area(S_0^n)=2\Area(\mb {RP}^n_0)$.
\end{enumerate}	
Here the first three items follow from the Properties (\ref{item:Ric geq 0}--\ref{item:cmc foliation}) of $M_k$ at the beginning of this section. For the last one, since $S_0^n\times \mb R/\{x\sim -x\}$ has a minimal foliation, it follows that 
\[  \omega_1(S_0^n\times \mb R/\{x\sim -x\}) \leq \Area(S_0^n)=2\Area(\mb {RP}^n_0). \]
On the other hand, as the limit contains $S_0^n \times (0,\infty)$, it follows that 
\[  \omega_1(S_0^n\times \mb R/\{x\sim -x\}) \geq \omega_1(S_0^n\times \mb R)= \Area(S_0^n).   \]
Hence the last property is proved.

	 The following result directly implies Corollary \ref{cor:one-sided}.
\begin{corollary}\label{cor:Nk}
For sufficiently large $k$, the first width $\omega_1(N_k)$ can only be realized by $\mb{RP}_0^n$ with multiplicity two. 
\end{corollary}
\begin{proof}
Notice that $N_k$ also satisfies the Frankel's property. Then by Marques-Neves \cite{MN16}, $\omega_1(N_k)$ is realized by the area of a connected, closed, embedded, minimal hypersurface $\Gamma_k$ with integer multiplicity $m_k$ and $\Index(\Gamma_k)\leq 1$. Observe that 
\begin{equation}\label{eq:limit of omegak:one-sided}
\lim_{k\to\infty}\omega_1(N_k)=\omega_1(S_0^n\times \mb R/\{x\sim -x\})=2\Area(\mb {RP}_0^n). 
\end{equation}
It follows that $\Area(\Gamma_k)$ is uniformly bounded. Then by compactness \cite{Sharp17}, $\Gamma_k$ converges to an embedded minimal hypersurface $\Gamma\subset S_0^n\times \mb R/\{x\sim -x\}$. Notice that such a space is foliated by minimal hypersurfaces. Thus $\Gamma$ is $S_0^n\times\{t\}$ or $\mb {RP}^n_0$. Recall that $\Gamma_k$ intersects $\mb {RP}^n_0$. Hence we conclude that the limit is $\mb {RP}^n_0$. Together with \eqref{eq:limit of omegak:one-sided}, we then have that $m_k\leq 2$.

\medskip
{\noindent\em  Case I: $m_k=1$.}
\medskip

Then $\Gamma_k$ locally smoothly converges to $2\mb{RP}^n_0$ away from at most one point. Then by the same argument as in Theorem \ref{thm:rigidity}, 
\[ \Area(\Gamma_k)>2\Area(\mb{RP}^n_0),   \]
which contradicts $\omega_1(N_k)\leq 2\Area(\mb{RP}^n_0)$. 

\medskip
{\noindent\em  Case II: $m_k=2$.}
\medskip

Then $\Gamma_k$ converges to $\mb{RP}^n_0$ with multiplicity one. By Allard's regularity, the convergence is smooth. Let $\wti\Gamma_k\in M_k$ be the double cover of $\Gamma_k$. Then $\wti \Gamma_k$ smoothly converges to $S_0^n$. This implies that $\wti \Gamma_k$ is identical to $S_0^n$. Hence $\Gamma_k$ is identical to $\mb{RP}^n_0$. This completes the proof of Corollary \ref{cor:Nk}.
	\end{proof}

\section{Upper bounds for the Hausdorff distance}\label{sec:hasdorff distance}
We use the notation in Section \ref{sec:blowup}. Recall that $\Sigma_k\subset M_k^{n+1}$ is a closed embedded minimal hypersurface such that the Morse index is bounded above by two and the area is uniformly bounded independent of $k\in \mb N$. As $k\to\infty$, $M_k$ converges locally smoothly to the product space $S_0^n\times \mb R$, and $\Sigma_k$ converges to a minimal $n$-sphere in the limit space of $M_k$, namely $S_0^n\times \mb R$. The convergence of $\Sigma_k$ is locally smooth away from at most two points according to the Morse index and area bounds. At a point near which the convergence is not smooth, the limit cones are either planes or catenoids by Proposition \ref{prop:limit cones}. For simplicity, we use $A$ for $A^{\Sigma_k}$ and sometimes omit the subscription $k$ when there is no ambiguity. Denote by $\nabla$ and $\wti \nabla$ the Levi-Civita connections of $\Sigma_k$ and $M_k$, respectively.

In this section, we prove the key height estimates, which says that the Hausdorff distance between $\Sigma_k$ and $S^n_0$ is bounded by a quantity associated with the catenoids arising from blowups. This result is essentially used in the previous section (see Theorem \ref{thm:Hausdorff distance0}).

Recall that by Remark \ref{rmk:structure of Sigmak}, $\Sigma_k$ has two connected components (denoted by $\Sigma_k^1$ and $\Sigma_k^2$) by removing one or two small catenoids. Denote by $y_k, z_k$ the centers and $r_k,\wti r_k$ the radii of links of such catenoids (we used $y_{k,j}$ and $r_{k,j}$ in Section \ref{sec:blowup} for general cases). Without loss of generality, we assume that $r_k\geq \wti r_k$.

Then by Proposition \ref{prop:2nd order estimates}, $\Sigma_k^2$ is a minimal graph over $\Sigma_k^1$. Let $\n$ be the unit normal vector field of $\Sigma_k^1$ , and let $\rho$ and $\wti \rho$ be the distance functions to $y_k$ and $z_k$ in $M_k$. Denote by 
\begin{equation}\label{eq:phi and eta}
 \bm\eta=\nabla \rho/|\nabla \rho|,\quad  \wti{\bm\eta}=\nabla \wti \rho/|\nabla \wti \rho|; \quad \phi=|\nabla \rho|, \quad \wti \phi=|\nabla \wti \rho|.
 \end{equation} 
Recall that $b_k$, defined in ~\eqref{eq:b_k}, is the distance between the two blowup points $y_k, z_k$. For any 
\begin{equation}\label{eq:good s}
\epsilon\geq s\geq 2b_k>0 \ \ \text{  or } \ \ \ b_k/2\geq s\geq 4r_k,
\end{equation}
we set  
\begin{equation}
\label{eq:gamma_s}
\gamma_{s}=\Sigma_k^1\cap\partial B(y_k,s;M_k).
\end{equation}

In the remaining part of this section, we assume that $R_k,\wti R_k\to 0$ are two sequence of real numbers satisfying $R_k/r_k\to \infty$ and $\wti R_k/\wti r_k\to\infty$. Moreover, we can also assume $r_k^{-1}(\Sigma_k\cap B(y_k,R_k;M_k)-y_k)$ (resp. $\wti r_k^{-1}(\Sigma_k\cap B(z_k,\wti R_k;M_k)-y_k)$) is arbitrarily close to the catenoid $\mc C\cap B_{R_k/r_k}(0)$ (resp. $\mc C\cap B_{\wti R_k/\wti r_k}(0)$) in the smooth topology. By Proposition \ref{prop:limit cones}, for sufficiently large $k$, $s^{-1}(\gamma_s-y_k)$ is very close to the unit $(n-1)$-sphere in the smooth topology. In particular, we have
\begin{equation}\label{eq:gammas round}  (1-\epsilon^2)\Omega_{n-1}s^{n-1}<|\gamma_s|<(1+\epsilon^2)\Omega_{n-1}s^{n-1},  
\end{equation}
where $\Omega_m$ is the volume of unit $m$-spheres.
Recall that $w_k$ is the graph function of $\Sigma_k^2$ over $\Sigma_k^1$.

\begin{lemma}\label{lem:laplace wk}
Given any $\delta>0$, then for all sufficiently large $k$, we have
\begin{equation}\label{eq:upper bound for laplace wk:general}
  \Delta w_k\leq  w_k+C(n)\delta\cdot \frac{w_k^3}{d^4_k(x)}
 \end{equation}
for $x\in\Sigma_k^1\setminus \big[B(y_k,R_k;M_k)\cup B(z_k,\wti R_k;M_k)\big]$.	
In particular, 
\begin{enumerate}
	\item for all large $k$ and $x\in \Sigma_k^1\cap A(y_k,R_k,\epsilon;M_k)\setminus B(z_k,b_k/2;M_k)$,
\begin{equation}\label{eq:upper bound for laplace wk}
 \Delta w_k\leq  w_k+\frac{1}{8}\cdot \frac{w_k^3}{|\rho(x)|^4} ;   
 \end{equation}
\item for all large $k$ and $x\in\Sigma_k^1\cap A(z_k,R_k,b_k/2;M_k)$,
\[ \Delta w_k\leq w_k+\frac{1}{8}\cdot \frac{w_k^3}{|\wti \rho(x)|^4} .   \]
\end{enumerate} 	
\end{lemma}

\begin{proof}
By Proposition \ref{prop:2nd order estimates}, $|A^{\Sigma_k}|$ has an upper bound over $\Sigma_k^1\cap A(y_k,s,2s;M_k)\setminus B(z_k,b_k/2;M_k)$, so the level sets of the distance to $\Sigma_k^1$ will form a desired foliation as in Appendix \ref{sec:graph functions}. By the Cauchy-Schwartz inequality, we have 
\begin{gather}\label{eq:mean inequ}
 \ 8|A||\nabla w_k|^2 \leq \frac{1}{2}|A|^2w_k+\frac{32|\nabla w_k|^4}{w_k};\quad |\nabla^2w_k||A|w_k\leq \delta|A|^2w_k+\frac{1}{4\delta}|\nabla ^2w_k|^2w_k.
\end{gather}
By plugging \eqref{eq:|A| small}, \eqref{eq:derivative of graphs} and \eqref{eq:mean inequ} into \eqref{eq:minimal graph equation} in Appendix \ref{sec:graph functions}, we can show that 
\[  \Delta w_k+|A|^2w_k\leq 3\delta w_k+|A|^2w_k+C(n)\delta\cdot \Big(\frac{w_k^3}{d^4_k(x)}+w_k\Big).\]
Here in~\eqref{eq:minimal graph equation}, we deal with the first and third terms by \eqref{eq:mean inequ}; for the second term in~\eqref{eq:minimal graph equation}, we used $|A|w_k\leq d_k|A|\cdot \frac{w_k}{d_k}<\delta^2$; in the fourth term, we used $w_k+|\nabla w_k|+|\nabla w_k|^3<3\delta$ by \eqref{eq:derivative of graphs}; the others can be bounded similarly; for instance, we have 
\[ |\nabla w_k|^2|\nabla^2w_k|<\Big(\frac{\delta w_k}{d_k}\Big)^2\cdot\frac{w_k}{d_k^2}<\delta\cdot \frac{w_k^3}{d_k^4} . \]
Hence the \eqref{eq:upper bound for laplace wk:general} is proved. 

Then for  $x\in\Sigma_k^1\cap  A(y_k,R_k,\epsilon;M_k)\setminus B(z_k,b_k/2;M_k)$, it follows that 
\[ d_k(x)\geq \frac{1}{4}\rho(x).  \]
Plugging this into \eqref{eq:upper bound for laplace wk}, we then have 
\begin{equation}
\Delta w_k\leq w_k+\frac{1}{8}\cdot \frac{w_k^3}{|\rho(x)|^4}.
\end{equation}
The last case can be proved similarly. This completes the proof of Lemma \ref{lem:laplace wk}.
\end{proof}

Inspired by Colding-Minicozzi \cite{CM02}*{(2.1)}, we define 
\begin{equation}\label{def:Ik}
\mc I_k(s)=\frac{1}{\Omega_{n-1} s^{n-1}}\int_{\gamma_{s}}w_k|\nabla \rho|\, d\mc H^{n-1}(x),
\end{equation}
\begin{equation}\label{def:tau}
\tau_k(s)=\frac{1}{\Omega_{n-1}}\int_{\gamma_{s}}\langle\nabla w_k,\bm\eta\rangle \ \ \ \text{ and } \ \ \ F_k(s)=\frac{n}{\Omega_{n-1}s^n}\int_{\gamma_{s}}(\phi^{-1}-\phi),
\end{equation}
where $\phi=|\nabla \rho|$ is defined in \eqref{eq:phi and eta}. Compared with \cite{CM02}, we introduce the weight $\phi=|\nabla \rho|$ in $\mc I_k$ as $\Sigma_k^1$ is not flat. Note that by Proposition \ref{prop:standard nabla w:high dim},
\begin{equation}\label{eq:estime of tau Rk}
	  \lim_{k\to\infty} \frac{\tau_k(R_k)}{r_k^{n-1}}=2. 
\end{equation}

\subsection{Inequalities for the first derivatives  }	  
In this subsection, we take the derivative for the averages of the height functions between two sheets. The coefficients of leading terms are the functions $\tau_k$ defined above.  	

\begin{proposition}\label{prop:computation}
	  	Using above notations, we then have	
	  	\begin{align*}
	  	 \mc I_k'(s)-\frac{\tau_k(s)}{ s^{n-1}}&=\frac{1}{\Omega_{n-1} s^{n-1}}\int_{\gamma_s}w_k\Big[\phi^{-1}\dv_{\Sigma_k}\wti\nabla \rho+\frac{1-n}{s}\phi\Big]\,d\mc H^{n-1}\\
	  	 &\leq \frac{n}{\Omega_{n-1}s^{n}}\int_{\gamma_s}(\phi^{-1}-\phi)w_k+C_0s\mc I_k(s),
	   	\end{align*} 	
where $C_0$ is a constant depending only on $n$ and can be changed from line to line.
\end{proposition}
 
\begin{proof}
Recall that $\phi = |\nabla \rho|$. A direct computation gives that	
	\begin{equation}\label{eq:divk nabla p}
	\langle \nabla \phi,\bm \eta\rangle+\phi\dv_{\gamma_s}\bm \eta =\langle \nabla \phi,\bm \eta\rangle+\phi\dv_{\Sigma_k}\bm \eta=\dv_{\Sigma_k}\phi\bm \eta=\dv_{\Sigma_k}\nabla \rho=\dv_{\Sigma_k}\wti \nabla \rho.
	\end{equation}
	Here we used that $\langle \nabla_{\bm \eta}\bm\eta,\bm\eta\rangle=0$ in the first equality; and the last one follows from the minimality of $\Sigma_k$. 	Then 
	by noting that $\frac{\partial}{\partial s} = \frac{\nabla \rho}{|\nabla\rho|^2} =\frac{\bm \eta}{\phi}$, we have
	\begin{align*}
	\mc I'_k(s)&=\frac{1-n}{\Omega_{n-1} s^n}\int_{\gamma_{s}}w_k\phi\, d\mc H^{n-1}+\frac{1}{\Omega_{n-1} s^{n-1}}\int_{ \gamma_{ s}}\langle \nabla (w_k\phi), \frac{\bm \eta}{\phi}\rangle+w_k\phi \dv_{\gamma_s}\big(\frac{\bm \eta}{\phi}\big)\,d\mc H^{n-1}\\
	&=\frac{1-n}{\Omega_{n-1} s^n}\int_{\gamma_s}w_k\phi\,d\mc H^{n-1}+\frac{1}{\Omega_{n-1} s^{n-1}}\int_{\gamma_s}\langle\nabla w_k,\bm\eta\rangle+w_k\big[ \langle \nabla \phi,\frac{\bm \eta}{\phi}\rangle+\dv_{\gamma_s}\bm \eta    \big]\,d\mc H^{n-1}\\
	&=\frac{1-n}{\Omega_{n-1} s^n}\int_{\gamma_s}w_k\phi\,d\mc H^{n-1}+\frac{1}{\Omega_{n-1} s^{n-1}}\int_{\gamma_s}\langle\nabla w_k,\bm\eta\rangle+w_k\phi^{-1}\dv_{\Sigma_k}\wti\nabla \rho\,d\mc H^{n-1},
	\end{align*}
	where the last equality is from \eqref{eq:divk nabla p}. Then we are going to prove the inequality. Recall that when restricting to $T(\partial B(p,s;M_k))$, we have
\begin{equation} \label{eq:hessian rho}
\wti\nabla^2\rho=\frac{1}{\rho}\delta_{ij}+\mc O(\rho).
\end{equation}
Therefore, we have that for $x\in\gamma_s$,
	\begin{align}\label{eq:div nabla rho}
	\dv_{\Sigma_k}\wti\nabla \rho&=\dv_{M_k}\wti\nabla \rho-\langle\wti\nabla_{\mf n}\wti\nabla \rho,\mf n\rangle\\
	&=\frac{n}{s}+\mc O(s)-\langle\wti\nabla_{ E}\wti\nabla \rho,E\rangle\langle \mf n,E\rangle^2\nonumber\\
	&\leq \frac{1}{s}(n-\phi^2)+C_0s.\nonumber
	\end{align}
In the above calculation, we used $E$ for the following expression,
\[  E:=\frac{\n-\langle \n,\wti \nabla \rho\rangle\wti\nabla \rho}{\sqrt{1-\langle \n,\wti \nabla \rho\rangle ^2}}; \]
this is the unit projection of $\n$ to $T(\partial B(p,s;M_k))$. 	
This finishes the proof of Proposition \ref{prop:computation}.
\end{proof}

In the next subsection, we will use the inequality in Proposition \ref{prop:computation} to bound $\mc I_k$, which requires the following results for $F_k(t)$. 

\begin{proposition}\label{prop:bound of F}
Suppose that $[R,s]\subset [R_k,b_k/2]$ or $[2b_k,\epsilon]$. Then	
\begin{align*}
\frac{d}{d s}\Big[s^{-n}\Big(|\Sigma_k^1\cap A(y_k,R,s;M_k)|+\frac{R}{n}\int_{\gamma_R}\phi \Big)\Big]
&=s^{-n}\int_{\gamma_s}(\frac{1}{\phi}-\phi)+\frac{1}{s^{n+1}}\int_{A(y_k,R,s;M_k)\cap \Sigma_k^1} \mc O(\rho^2).
\end{align*}
For all sufficiently large $k$,
\[ \int_{R_k}^{b_k/2}F_k(t)dt+\int_{2b_k}^{\epsilon} F_k(t)dt<\epsilon. \]
\end{proposition}

\begin{proof}
By the divergence theorem,
\begin{equation}
\int_{A(y_k,R,s;M_k)\cap \Sigma_k^1}\Delta\rho^2\label{eq:Delta p-s}= \int_{\gamma_s}2s\phi -\int_{\gamma_R}2R\phi .
\end{equation}
Since $\Delta \rho^2= 2n+\mc O(\rho^2)$, we have
\[ |\Sigma_k^1\cap A(y_k,R,s;M_k)|+\frac{R}{n}\int_{\gamma_R}\phi = \frac{s}{n} \int_{\gamma_s} \phi + \int_{A(y_k,R,s;M_k)\cap \Sigma_k^1} \mc O(\rho^2). \]
Finally note that
\[ \frac{d}{ds}\Big|\Sigma_k^1\cap A(y_k,R,s;M_k)\Big|=\int_{\gamma_s}\phi^{-1}.\]
The first identity follows by plugging all the above identities. Then we have
\[ 
\begin{aligned}
\int_{R_k}^{b_k/2}F_k(t)\,dt
& \leq  \frac{n}{\Omega_{n-1}(b_k/2)^{n}}\Big(|\Sigma_k^1\cap A(y_k,R_k,b_k/2;M_k)|+\frac{R_k}{n}\int_{\gamma_{R_k}}\phi \Big)\\
& -\frac{1}{\Omega_{n-1}R_k^{n-1}}\int_{\gamma_{R_k}}\phi+C_0b_k^2.
\end{aligned}
 \]
 Recall that by \eqref{eq:gammas round}, for all sufficiently large $k$, we have
 \[  (1-\frac{\epsilon}{10})\Omega_{n-1}s^{n-1}\leq |\gamma_s|\leq (1+\frac{\epsilon}{10})\Omega_{n-1}s^{n-1} ; \ \ \ 1\leq \phi^{-1}\leq 1+\frac{\epsilon}{10} .\]
 Then by the co-area formula, we have
 \begin{align*}
   &\ \ \ \frac{n}{\Omega_{n-1}(b_k/2)^{n}}\Big(|\Sigma_k^1\cap A(y_k,R_k,b_k/2;M_k)|+\frac{R_k}{n}\int_{\gamma_{R_k}}\phi \Big)\\
   & \leq\frac{n}{\Omega_{n-1}(b_k/2)^{n}}\, \cdot \Omega_{n-1} \Big[\int_{R_k}^{b_k/2}(1+\frac{\epsilon}{10})s^{n-1}\,ds+\frac{R_k}{n}(1+\frac{\epsilon}{10})R_k^{n-1}  \Big]= 1+\frac{\epsilon}{10}.
\end{align*}
On the other hand,
  \[ \frac{1}{\Omega_{n-1}R_k^{n-1}}\int_{\gamma_{R_k}}\phi\geq 1-\frac{\epsilon}{5}.  \]
Combining them together, we then have 
 \[ \int_{R_k}^{b_k/2}F_k(t)\,dt\leq \frac{3\epsilon}{10}+C_0\epsilon^2<\frac{\epsilon}{2}.   \]
Similarly,
 \[ \int_{2b_k}^{\epsilon} F_k(t)\,dt<\frac{\epsilon}{2}.   \]
Hence Proposition \ref{prop:bound of F} is proved.
\end{proof}

\subsection{Hausdorff distance upper bounds in dimension three}
In this part, we focus on the case $n=2$ and prove the upper bounds of the Hausdorff distance between $\Sigma_k$ and $S_0^2$. Recall that $\Sigma_k^1$ and $\Sigma_k^2$ are jointed by one or two small catenoids with radii $r_k$ and $\wti r_k$; see Remark \ref{rmk:structure of Sigmak}. Without loss of generality, we assume that $r_k\geq\wti  r_k$. Recall that $\Sigma_k^2$ is a graph over $\Sigma_k^1$ with graph function $w_k$.

To prove the desired bound, we will derive several monotonicity formulas associated with the average of $w_k$. Then the smooth convergence on $\partial B(p,\epsilon;M_k)$ will give the desired upper bound.

Now let 
\begin{gather}\label{eq:wti Ik}
\wti{\mc I}_k(s)=\mc I_k(s)-3r_k\log \frac{s}{r_k}-2sr_k\log\frac{s}{r_k}-10\Big(\int_{[R_k,s]\setminus [b_k/2,2b_k]}F_k(t)\,dt\Big) \cdot r_k\log \frac{s}{r_k};\\
\wti \tau_k(s)=\tau_k(s)+\frac{r_k^2}{2s}-sr_k\log\frac{s}{r_k}.\label{eq:wti tauk}
\end{gather}
Then by Proposition \ref{prop:bound of F},
\begin{equation}\label{eq:Ik and wti Ik}
\mc I_k(s)< \wti {\mc I}_k(s)+4r_k\log \frac{s}{r_k}. 
\end{equation}
Recall that $R_k/r_k\to +\infty$ and $r_k^{-1}(\Sigma_k\cap B(y_k,R_k;M_k)-y_k)$ is arbitrarily close to the catenoid $\mc C\cap B_{R_k/r_k}(0)$ in the smooth topology. Then by Item \eqref{item:h bound}in Appendix \ref{sec:catenoids},
\begin{equation}\label{eq:Ik Rk}
\mc I_k(R_k)\leq (1+\frac{1}{100})\cdot 2r_k\int_1^{R_k/r_k}\frac{ds}{\sqrt{s^2-1}}<(1+\frac{1}{100})\cdot 2r_k\log\frac{2R_k}{r_k}<3r_k\log\frac{R_k}{r_k}.
\end{equation}
Then by \eqref{eq:estime of tau Rk}, \eqref{eq:wti Ik} and \eqref{eq:Ik Rk},
\begin{equation}\label{eq:wtiIk Rk}
\wti {\mc I}_k(R_k)<0,  \ \ \ \tau_k(R_k)<(2+\frac{1}{10})r_k.
\end{equation}

Now we are ready to show that $\wti {\mc I}_k(s)$ is decreasing.
\begin{proposition}\label{prop:bound wk}
	 $\wti {\mc I}_k(s)$ and $\wti\tau_k(s)$ are decreasing on $[R_k,b_k/2]$ (resp. $[R_k,\epsilon]$)  if  $b_k\neq 0$ (resp. $b_k=0$). It follows that 
	\begin{gather*}
	\mc I_k(s)< 4r_k\log\frac{s}{r_k},  \ \  \ 	\tau_k(s)< 3r_k-\frac{r_k^2}{2s}+sr_k\log\frac{s}{r_k} \ \ \text{ for } \  s\in[R_k,b_k/2].
\end{gather*}
\end{proposition}
\begin{proof}
	We first assume that $b_k\neq 0$. 
	To prove that $\wti{\mc I}_k$ and $\wti\tau_k$ are decreasing on $[R_k,b_k/2]$, let 
\[ s_1:=\sup\{ s\in (R_k,b_k/2): \wti \tau'_k(t)\leq 0 ,\  \wti{\mc I}_k'(t)\leq 0  \ \text{ for  all } \ t\in[R_k,s) \} .\]	
We claim that $s_1=b_k/2$. Suppose on the contrary that $s_1<b_k/2$. Then $\wti{\mc I}_k$ and $\wti\tau_k$ are decreasing on $[R_k,s_1]$, which together with \eqref{eq:wtiIk Rk} implies that 
\begin{gather*}
 \wti{\mc I}_k(s_1)\leq \wti{\mc I}_k(R_k)<0 ;\ \ \	\wti \tau_k(s_1)\leq \wti\tau_k(R_k)< \frac{5r_k}{2}.
 	\end{gather*}
By \eqref{eq:wti tauk} and \eqref{eq:Ik and wti Ik}, it follows that
\begin{gather*}
 \mc I_k(s_1)< \wti {\mc I}_k(s_1)+4r_k\log\frac{s_1}{r_k}< 4 r_k\log \frac{s_1}{r_k} ;\\
 \tau_k(s_1)= \wti \tau_k(s_1)-\frac{r_k^2}{2s_1}+s_1r_k\log \frac{s_1}{r_k}< 3r_k -\frac{r_k^2}{2s_1}+s_1r_k\log \frac{s_1}{r_k}.
 \end{gather*} 
 Then by Proposition \ref{prop:computation}, 
 \begin{align*}
   \mc I_k'(s_1)&< \frac{\tau_k(s_1)}{s_1}+ 5 r_k\log\frac{s_1}{r_k} \cdot F_k(s_1)+C_0s_1r_k\log\frac{s_1}{r_k}\\
   &< \frac{3r_k}{s_1}+2r_k\log \frac{s_1}{r_k}+5r_k\log\frac{s_1}{r_k} \cdot F_k(s_1).
   \end{align*}
  Here in the first inequality, we used that by Proposition \ref{prop:2nd order estimates} \eqref{item:harnack} and \eqref{eq:gammas round}, for $x\in \gamma_{s_1}$,
  \begin{equation}\label{eq:wk bound on s1}
  w_k(x)\leq (1+1/100)^2\mc I_k(s_1)<5 r_k\log\frac{s_1}{r_k} .  
  \end{equation}
Then by \eqref{eq:wti Ik}, it follows that 
 \begin{equation}\label{eq:wtiIk'<0}
  \wti {\mc I}_k'(s_1) <\mc I_k'(s_1)-\frac{3r_k}{s_1}-2r_k\log \frac{s_1}{r_k}-10F_k(s_1)\cdot r_k\log \frac{s_1}{r_k} <0. 
 \end{equation} 
 On the other hand, by the divergence theorem, the co-area formula, ~\eqref{eq:upper bound for laplace wk} and ~\eqref{eq:wk bound on s1}, 
 \begin{align}
  \tau_k'(s_1)=\frac{1}{2\pi}\int_{\gamma_{s_1}}\phi^{-1}\Delta w_k &\leq \frac{1}{2\pi}\int_{\gamma_{s_1}}\phi^{-1}\Big(w_k+\frac{w_k^3}{8s_1^4}\Big)\\
 &<  2s_1\cdot  \Big( 5 r_k\log\frac{s_1}{r_k}+\frac{1}{8 s_1^4} 5^3 r_k^3\big(\log \frac{s_1}{r_k}\big)^3 \Big) < r_k\log\frac{s_1}{r_k}+\frac{r_k^2}{2s_1^2}. \nonumber
  \end{align}
  Here the first inequality follows from \eqref{eq:upper bound for laplace wk}; the second one follows from \eqref{eq:wk bound on s1}, \eqref{eq:gammas round} and $\phi^{-1}|_{\gamma_{s_1}}\leq 1+\epsilon$; in the last one, we used that as $k\to\infty$,
  \[  s_1\to 0;\ \ \frac{s_1}{r_k}\to\infty; \ \  \frac{r_k}{s_1}(\log\frac{s_1}{r_k})^3 \to 0. \]
 Together with ~\eqref{eq:wti tauk}, it follows that 
 \[ \wti\tau_k'(s_1)<\tau_k'(s_1)-\frac{r_k^2}{2s_1^2} -r_k\log \frac{s_1}{r_k}<0.   \]
 Combining with ~\eqref{eq:wtiIk'<0}, this contradicts the choice of $s_1$.
 
 \medskip
 If $b_k=0$, then the same argument above gives that $\wti{\mc I}_k$ and $\wti\tau_k$ are decreasing on $[R_k,\epsilon]$. This finishes the proof of Proposition \ref{prop:bound wk}.
 \end{proof}

 Note that $\mc I_k(s)$ is not well-defined in a small neighborhood of $b_k$. Here we use Proposition  \ref{prop:2nd order estimates} and the divergence theorem to jump over such a small interval.
 
\begin{proposition}
\label{prop:bound wk both}
Suppose that $b_k\neq 0$. Then $\wti{\mc I}_k-3r_k\log\frac{s}{r_k}$ and $\wti\tau_k$ are decreasing on $[2b_k,\epsilon]$. It follows that 
\[  \mc I_k(s)<7r_k\log\frac{s}{r_k},  \ \ \ \ 
\tau_k(s)< 6 r_k+sr_k\log\frac{s}{r_k} \ \ \text{ for  } \ s\in [2b_k,\epsilon]. \] 
In particular, we conclude that for sufficiently large $k$,
\[ \max_{x\in \Sigma_k^1\cap \partial B(y_k,\epsilon;M_k) }w_k < \frac{15}{2}r_k\log\frac{\epsilon}{r_k}.  \]
\end{proposition}
\begin{proof}
By applying Proposition \ref{prop:2nd order estimates} \eqref{item:harnack}, and then Proposition \ref{prop:bound wk}, we have
\begin{equation}\label{eq:boun wk around b_k}
 \max\{w_k; x\in \Sigma_k^1\cap A(y_k,b_k/2,2b_k;M_k)\setminus B(z_k,b_k/4;M_k)  \}\leq (1+\delta) \mc I_k(\frac{b_k}{2}) < \frac{9}{2}r_k\log\frac{b_k}{2r_k}. 
 \end{equation}
Then by \eqref{eq:gammas round},
\begin{equation}\label{eq:Ik on 2bk}
\mc I_k(2b_k)< 5 r_k\log \frac{b_k}{2r_k}.
\end{equation}
Moreover, by the divergence theorem, 
\begin{align*}
 \tau_k(2b_k)-\tau_k(b_k/2)-\tau_k(b_k/4;z_k)&= \frac{1}{2\pi}\int_{\Sigma_k^1\cap A(y_k,b_k/2,2b_k;M_k)\setminus B(z_k,b_k/4;M_k)} \Delta w_k\,dx\\   
& \leq \frac{1}{2\pi}\int_{\Sigma_k^1\cap A(y_k,b_k/2,2b_k;M_k)\setminus B(z_k,b_k/4;M_k)}w_k+\frac{w_k^3}{8\rho^4}\,dx \\ 
&< 10b_k^2r_k\log \frac{b_k}{r_k}+ \frac{5^3r_k^3}{b_k^2}\Big(\log\frac{b_k}{r_k}\Big)^3\\
 &<\frac{1}{4} b_kr_k\log\frac{b_k}{2r_k}+\frac{r_k^2}{2b_k}.
 \end{align*}
Here we used the notation
\[\tau_k(s;z_k)=\frac{1}{2\pi}\int_{\Sigma_k^1\cap\partial B(z_k,s;M_k)}\langle \nabla w_k, \bm \eta_z\rangle , \]
and $\bm \eta_z $ is the unit normal to $\partial B(z_k,s;M_k)\cap \Sigma_k^1$. The first inequality is from \eqref{eq:upper bound for laplace wk}; we used \eqref{eq:boun wk around b_k} in the second one, and the area upper bound of $\Sigma_k^1\cap A(y_k,b_k/2,2b_k;M_k)\setminus B(z_k,b_k/4;M_k)$ is from the fact that $b_k^{-1}(\Sigma_k^1-y_k)$ is very close to a hyperplane. Recall that $\Sigma_k$ is very close to a catenoid with radius $\wti r_k\leq r_k$ around $z_k$. Then by the same argument as in Proposition \ref{prop:bound wk}, we also have 
\[ \tau_k(b_k/4;z_k)< 3\wti r_k+\frac{1}{4}b_k\wti r_k\log\frac{b_k}{4\wti r_k}\leq 3r_k+\frac{1}{4}b_kr_k\log\frac{b_k}{4r_k}.  \]
Here the inequality is from $\wti r_k\leq r_k$, and the monotonicity of $r\mapsto r\log\frac{b_k}{4r}$ when $b_k/4r_k > e^{-1}$.
It follows that 
\begin{align}\label{eq:tau 2bk}
 \tau_k(2b_k)&< \tau_k(b_k/2)+\tau_k(b_k/4;z_k)+\frac{1}{4} b_kr_k\log\frac{b_k}{2r_k}+\frac{r_k^2}{2b_k}\\
 &<  3r_k-\frac{r_k^2}{b_k}+ \frac{1}{2}b_kr_k\log\frac{b_k}{2r_k}+3r_k+ \frac{1}{4}b_kr_k\log\frac{b_k}{4r_k}+\frac{1}{4} b_kr_k\log\frac{b_k}{2r_k}+\frac{r_k^2}{2b_k}\nonumber\\
 &< 6r_k-\frac{r_k^2}{4b_k}+b_kr_k\log\frac{b_k}{r_k}.\nonumber
 \end{align}
 
In the next, we are going to prove that $\wti{\mc I}_k(s)-3r_k\log\frac{s}{r_k}$ and $\wti\tau_k$ are decreasing. Let 
\[ s_2:=\sup\{ s\in [2b_k,\epsilon]: \wti \tau'_k(t)\leq 0 ,\  \wti{\mc I}_k'(t)\leq \frac{3r_k}{t}  \ \text{ for  all } \ t\in[2b_k,s] \} .\]
It suffices to prove that $s_2=\epsilon$. Suppose on the contrary that $s_2<\epsilon$. Then 
$\wti{\mc I}_k-3r_k\log\frac{s}{r_k}$ and $\wti\tau_k$ are decreasing on $[2b_k,s_2]$, and this implies that 
\begin{gather*}
\wti{\mc I}_k(s_2)-3r_k\log\frac{s_2}{r_k}\leq \wti{\mc I}_k(2b_k)-3r_k\log\frac{2b_k}{r_k}< \mc I(2b_k)-6r_k\log \frac{2b_k}{r_k}<0,\,  \text{by \eqref{eq:Ik on 2bk}};\\
	\wti \tau_k(s_2)\leq \wti\tau_k(2b_k)= \tau_k(2b_k)+\frac{r_k^2}{4b_k}-2b_kr_k\log \frac{2b_k}{r_k}< 6r_k,\, \text{by \eqref{eq:tau 2bk}} .
\end{gather*}
Together with \eqref{eq:Ik and wti Ik} and \eqref{eq:wti tauk}, we then have 
\begin{gather}
\mc I_k(s_2)<\wti {\mc I}_k(s_2)+4r_k\log\frac{s_2}{r_k}< 7 r_k\log \frac{s_2}{r_k}\label{eq:Ik s2} ;\\
\tau_k(s_2)= \wti \tau_k(s_2)-\frac{r_k^2}{2s_2}+s_2r_k\log \frac{s_2}{r_k}< 6r_k-\frac{r_k^2}{2s_2}+s_2r_k\log \frac{s_2}{r_k}.\nonumber
\end{gather}
Recall that by Proposition \ref{prop:computation}, 
\begin{align*}
\mc I_k'(s_2)&\leq \frac{\tau_k(s_2)}{s_2}+ 9r_k\log\frac{s_2}{r_k} \cdot F_k(s_2)+C_0s_2r_k\log\frac{s_2}{r_k},\\
&< \frac{6r_k}{s_2}+2r_k\log \frac{s_2}{r_k}+9r_k\log\frac{s_2}{r_k} \cdot F_k(s_2),
\end{align*}
where in the first inequality, we used that for all $x\in \gamma_{s_2}$, 
\begin{equation}\label{eq:wk s2}
 w_k(x)\leq (1+1/10)\mc I_k(s_2)<9r_k\log \frac{s_2}{r_k}. 
 \end{equation}
 By \eqref{eq:wti Ik}, it follows that 
\begin{equation}\label{eq:wtiIk'<0:after bk}
\wti {\mc I}_k'(s_2) <\mc I_k'(s_2)-\frac{3r_k}{s_2}-2r_k\log \frac{s_2}{r_k}-10F_k(s_2)\cdot r_k\log \frac{s_2}{r_k} <\frac{3r_k}{s_2}. 
\end{equation} 
On the other hand, by the divergence theorem and the co-area formula,
\begin{align*} \tau_k'(s_2)&=\frac{1}{2\pi}\int_{\gamma_{s_2}}\phi^{-1}\Delta w_k \leq \frac{1}{2\pi}\int_{\gamma_{s_2}}\phi^{-1}\Big(w_k+\frac{w_k^3}{8s_2^4}\Big)\,dx\\
&\leq 4s_2\cdot 9 r_k\log\frac{s_2}{r_k}+9^3\frac{r_k^3}{s_2^3}(\log \frac{s_2}{r_k})^3<r_k\log\frac{s_2}{r_k}+\frac{r_k^2}{2s_2^2}. 
\end{align*}
Here the first inequality is from \eqref{eq:upper bound for laplace wk}; the second one follows from \eqref{eq:Ik s2}, \eqref{eq:wk s2}, \eqref{eq:gammas round} and $\phi^{-1}|_{\gamma_{s_1}}\leq 1+\epsilon$; in the last one, we used that as $k\to\infty$,
\[  s_2\to 0;\ \ \frac{s_2}{r_k}\to\infty; \ \  \frac{r_k}{s_2}(\log\frac{s_2}{r_k})^3\to 0. \]
Then \eqref{eq:wti tauk} can be applied to get
\[ \wti\tau_k'(s_2)<\tau_k'(s_2)-\frac{r_k^2}{2s_2^2} -r_k\log \frac{s_2}{r_k}<0.   \]
Combining with \eqref{eq:wtiIk'<0:after bk}, it contradicts the choice of $s_2$.

Then the bound of $w_k$ on $\Sigma_k^1\cap \partial B(y_k,\epsilon;M_k)$ follows from Proposition \ref{prop:2nd order estimates}(3).  Hence Proposition \ref{prop:bound wk both} is proved.
\end{proof}

Recall that $\Sigma_k$ smoothly converges to $2S_0^2$ outside at most two balls. In particular, $\Sigma_k^1$ and $\Sigma_k^2$ are graphs over $S_0^2$ outside such two balls; see Remark \ref{rmk:structure of Sigmak}. Observe that the normalization of graph functions of $\Sigma_k^1$ and $\Sigma_k^2$ can both give bounded Jacobi fields, which would be smooth across the singularities. Then $w_k$ are equivalent to the difference of the two graph functions. By the maximum principle, such two graph functions should have opposite signs. Then the upper bound for $w_k$ implies the Hausdorff distance between $\Sigma_k$ and $S_0^2$. 
\begin{theorem}\label{thm:hausdorff distance}
	For all sufficiently large $k$,
	\[  \max_{\Sigma_k}\dist_{M_k}(x,S^2_0)\leq 8r_k|\log r_k|.\]
\end{theorem}
\begin{proof}
Recall that $\Sigma_k$ locally smoothly converges locally smoothly to $2\cdot S^2_0$ away from $\mc W$ which contains at most two points. Then for any compact set $\Omega\subset S_0^2\setminus \mc W$, $\Sigma_k$ can be decomposed into two minimal graph functions in the neighborhood of $\Omega$ for sufficiently large $k$. Denote by $u_k^1$ and $u_k^2$ the graph functions. Note that $u_k^1$ and $u_k^2$ are defined on any compact domain in $S_0^2\setminus \mc W$ by taking large $k$. Denote by 
\[\lambda_{k}(x,s)=\max_{\partial  B_s(x)\cap S_0^2}\{-u_k^1,u_k^2\} \ \ \text{ and } \ \ \Lambda_{k,s}=\max_{x\in\mc W}\lambda_k(x,s).\]
Then by a standard argument (see \cite{Si87}), $u_k^1/\Lambda_{k, \epsilon}$ and $u^2_k/\Lambda_{k,\epsilon}$ locally smoothly converges to $c_1$ and $c_2$, which are Jacobi fields of $S_0^2\setminus \mc W$. By the minimal foliation argument, $c_1$ and $c_2$ are bounded, which implies that they can be extended smoothly across $\mc W$. Therefore, $c_1$ and $c_2$ are constants. Moreover, by the definition of $\Lambda_{k,\epsilon}$, $c_1= -1$ or $c_2=1$. Without loss of generality, we assume that $c_2=1$. Note that $\Sigma_k$ does not lie in one side of $S_0^2$, which implies that $c_1\leq 0$. Then by the smooth convergence of $\Sigma_k$ in $B_{2\epsilon}(\mc W)\setminus B_{\epsilon/2}(\mc W)$, for sufficiently large $k$,
	\[ \lim_{k\to\infty }\max_{\Sigma_k^1\cap B_\epsilon(\mc W)}\frac{w_k}{\Lambda_{k,\epsilon}}=\lim_{k\to\infty}\max_{\partial  B_\epsilon(\mc W)\cap \Sigma}\frac{u_k^2-u_k^1}{\Lambda_{k,\epsilon}}= 1-c_1\geq 1.  \]
By Propositions \ref{prop:bound wk} and \ref{prop:bound wk both}, for sufficiently large $k$,
\[ \max_{x\in \Sigma_k^1\cap \partial B_{\epsilon}(\mc W) }w_k < (\frac{15}{2}+\frac{1}{100})r_k\log\frac{\epsilon}{r_k}.  \]
Thus we conclude that for all large $k$,
	\[ \Lambda_{k,\epsilon}<(\frac{15}{2} +\frac{1}{50})r_k|\log r_k|.  \]
	By the minimal foliation argument inside $B_{\epsilon}(p)$ and Harnack inequalities outside such a ball, 
	\[ \max_{x\in\Sigma_k}\dist_{M}(x,S_0^2)<8r_k|\log r_k|.   \]
This completes the proof of Lemma \ref{thm:hausdorff distance}.
	\end{proof}

\subsection{Upper bounds for Hausdorff distance in high dimensions}
In this part, we restrict $3\leq n\leq6$ and prove the upper bounds of the Hausdorff distance between $\Sigma_k$ and $S_0^{n}$. Recall that $\Sigma_k^1$ and $\Sigma_k^2$ are jointed by one or two small catenoids with radii $r_k$ and $\wti r_k$; see Remark \ref{rmk:structure of Sigmak}. Without loss of generality, we assume that $r_k\geq\wti  r_k$. Recall that $\Sigma_k^2$ is a graph over $\Sigma_k^1$ with graph function $w_k>0$.

To prove the desired bound, we start at one catenoid and derive new monotonicity formulas. Then the smooth convergence on $\partial B(p,\epsilon;M_k)$ will give the desired upper bound. Now let 
\begin{gather}\label{eq:wti Ik:high}
\wti{\mc I}_k(s)=\mc I_k(s)+4s^{-\frac{1}{2}}r_k^{\frac{3}{2}}-sr_k-10r_k\int_{[R_k,s]\setminus [b_k/2,2b_k]} F_k(t)\,dt;\\
\wti \tau_k(s)=\tau_k(s)-9s^nr_k-r_k^{\frac{3}{2}}s^{n-\frac{5}{2}}.\label{eq:wti tauk:high}
\end{gather}
Then by \eqref{prop:bound of F},
\begin{equation}\label{eq:Ik and wti Ik:high}
\mc I_k(s)\leq \wti {\mc I}_k(s)+\frac{r_k}{20}, 
\end{equation}
Recall that $R_k/r_k\to 0$ and $r_k^{-1}(\Sigma_k\cap B(y_k,R_k;M_k)-y_k)$ is arbitrarily close to the catenoid $\mc C\cap B_{R_k/r_k}(0)$ in the smooth topology. Then by Item \eqref{item:h bound} in Appendix \ref{sec:catenoids},
\begin{equation}\label{eq:Ik Rk:high}
\mc I_k(R_k)\leq (1+\frac{1}{100})\cdot 2r_k\int_1^{\infty}\frac{ds}{\sqrt{s^{2(n-1)}-1}}<\frac{27}{10}r_k.
\end{equation}
And by \eqref{eq:estime of tau Rk}, \eqref{eq:wti Ik:high} and \eqref{eq:Ik Rk:high},
\begin{equation}\label{eq:wtiIk Rk:high}
\wti {\mc I}_k(R_k)<\frac{14}{5}r_k,  \ \ \ \tau_k(R_k)<(2+\frac{1}{10})r_k^{n-1}.
\end{equation}

Now we are ready to show that $\wti {\mc I}_k(s)$ is decreasing in two disjoint intervals. 
\begin{proposition}\label{prop:bound wk:high}
 $\wti {\mc I}_k(s)$ and $\wti\tau_k(s)$ are decreasing on $[R_k,b_k/2)$ (resp. $[R_k,\epsilon]$) if $b_k\neq 0$ (resp. $b_k=0$). It follows that 
	\begin{gather*}
	\mc I_k(s)< 3r_k,  \ \  \ 	\tau_k(s)< 3r_k^{n-1}+9s^nr_k+r_k^{\frac{3}{2}}s^{n-\frac{5}{2}} \ \ \text{ for } \  s\in[R_k,b_k/2].
	\end{gather*}
\end{proposition}

\begin{proof}[Proof of Proposition \ref{prop:bound wk:high}]
Suppose $b_k\neq 0$. To prove that $\wti{\mc I}_k$ and $\wti\tau_k$ are decreasing on $[R_k,b_k/2]$, we let 
	\[ s_1:=\sup\{ s\in (R_k,b_k/2): \wti \tau'_k(t)\leq 0 ,\  \wti{\mc I}_k'(t)\leq 0  \ \text{ for  all } \ t\in[R_k,s) \} .\]	
 We claim that $s_1=b_k/2$. Suppose on the contrary that $s_1<b_k/2$. Then it follows that $\wti{\mc I}_k(s)$ and $\wti\tau_k(s)$ are decreasing on $[R_k,s_1]$. Hence by \eqref{eq:wtiIk Rk:high},
	\begin{gather*}
	\wti {\mc I}_k(s_1)\leq \wti{\mc I}_k(R_k)<\frac{14r_k}{5};\ \ \ 	\wti\tau_k(s_1)\leq \wti\tau_k(R_k)< \frac{5r_k^{n-1}}{2}.
	\end{gather*}
 Together with \eqref{eq:Ik and wti Ik:high} and \eqref{eq:wti tauk:high}, we have that
	\begin{gather}\label{eq:Ik s1:high}
	\mc I_k(s_1)\leq  \wti {\mc I}_k(s_1)+\frac{r_k}{20}  < \frac{29r_k}{10} ;\\
	\tau_k(s_1)= \wti \tau_k(s_1)+9s_1^nr_k+s_1^{n-\frac{5}{2}}r_k^{\frac{3}{2}} < \frac{5}{2}r_k^{n-1}+ 9s_1^nr_k+s_1^{n-\frac{5}{2}}r_k^{\frac{3}{2}}.\nonumber
	\end{gather}
Note that the first inequality together with \eqref{eq:gammas round} and Proposition \ref{prop:2nd order estimates} \eqref{item:harnack} implies that  for $x\in\gamma_{s_1}$,
\begin{equation}\label{eq:wk on s1:high}
w_k(x)\leq (1+1/100)\mc I_k(s_1)< 3r_k.
\end{equation}
 Then by Proposition \ref{prop:computation}, 
\begin{align*}
\mc I_k'(s_1)&\leq \frac{\tau_k(s_1)}{s_1^{n-1}}+    \frac{n}{\Omega_{n-1}s_1^{n}}\int_{\gamma_{s_1}}(\phi^{-1}-\phi)w_k+C_0s_1\mc I_k(s_1)\\
&\leq \frac{5}{2}\Big(\frac{r_k}{s_1}\Big)^{n-1}+ 9s_1r_k+s_1^{-\frac{3}{2}}r_k^{\frac{3}{2}}+3r_kF_k(s_1)+C_0s_1\cdot 3r_k\\
&< 2\Big(\frac{r_k}{s_1}\Big)^{\frac{3}{2}}+r_k+3r_kF_k(s_1),
\end{align*}
where the second inequality is from \eqref{eq:wk on s1:high}; the last one follows from $n-1\geq 2$, $s_1\to 0$ and $r_k/s_k\to 0$ as $k\to\infty$. Together with \eqref{eq:wti Ik:high}, we then have 
\begin{equation}\label{eq:wti Ik'<0}
 \wti{\mc I}_k'(s_1)= \mc I'_k(s_1)-2\Big(\frac{r_k}{s_1}\Big)^{\frac{3}{2}}-r_k- 10r_kF_k(s_1)< 0. 
 \end{equation}
On the other hand, by the divergence theorem and the co-area formula,
	\begin{align*}
\tau_k'(s_1)=\frac{1}{\Omega_{n-1}}\int_{\gamma_{s_1}}\phi^{-1}\Delta w_k&\leq \frac{1}{\Omega_{n-1}}\int_{\gamma_{s_1}}\phi^{-1}\Big(w_k+\frac{w_k^3}{8s_1^4}\Big)\\
&< (1+\frac{1}{100})^2\cdot \frac{29}{10}s_1^{n-1}r_k+\frac{3^3}{8}\cdot 4s_1^{n-5}r_k^3\\
&<  3s_1^{n-1}r_k+r_k^2s_1^{n-4}.
\end{align*}
Here the first inequality follows from \eqref{eq:upper bound for laplace wk}; in the second one, we used \eqref{eq:gammas round}, \eqref{eq:Ik s1:high}, \eqref{eq:wk on s1:high} and $\phi^{-1}\Big|_{\gamma_{s_1}}<1+\frac{1}{100}$; the last one is from $r_k/s_1\to0$. Together with \eqref{eq:wti tauk:high},
\[\wti\tau_k'(s_1)=\tau_k'(s_1) -9ns_1^{n-1}r_k-{(n-\frac{5}{2})}r_k^{\frac{3}{2}}s_1^{n-\frac{7}{2}}<0. \]
Combining with \eqref{eq:wti Ik'<0}, this contradicts the choice of $s_1$.

If $b_k=0$, then the same argument above gives that $\wti{\mc I}_k$ and $\wti\tau_k$ are decreasing on $[R_k,\epsilon]$. Hence  Proposition \ref{prop:bound wk:high} is proved.
\end{proof}

Note that $\mc I(s)$ is not well-defined in a small neighborhood of $b_k$. Here we use Proposition \ref{prop:2nd order estimates} and the divergence theorem to jump over such a small interval.
\begin{proposition}\label{prop:Ik bound pass zk}
	Suppose that $b_k\neq 0$. Then $\wti{\mc I}_k$ and  $\wti\tau_k$ are decreasing on $[2b_k,\epsilon]$. It follows that 
	\begin{gather*}
	\mc I_k(s)< 4r_k,  \ \ \ \ 
	\tau_k(s)<  6r_k^{n-1}+9s^nr_k+r_k^{\frac{3}{2}}s^{n-\frac{5}{2}} \ \ \text{ for  } \ s\in [2b_k,\epsilon].
	\end{gather*}
	In particular, we conclude that for sufficiently large $k$,
	\[ \max_{x\in \Sigma_k^1\cap \partial B(y_k,\epsilon;M_k) }w_k <\frac{9}{2}r_k.  \]
\end{proposition}
\begin{proof}
By applying Proposition \ref{prop:2nd order estimates} \eqref{item:harnack}, and then Proposition \ref{prop:bound wk:high}, together with \eqref{eq:gammas round}, we have
	\begin{equation}\label{eq:bound wk on annuli:high}
	 \max\{w_k; x\in \Sigma_k^1\cap A(y_k,b_k/2,2b_k;M_k)\setminus B(z_k,b_k/4;M_k)  \}\leq (1+\delta)\mc I_k\big(\frac{b_k}{2}\big) \leq \frac{10}{3}r_k,
	 \end{equation}
as well as 
	\begin{equation}\label{eq:Ik on 2bk:high}
	\mc I_k(2b_k)< \frac{11r_k}{3}.
	\end{equation}
	Moreover, by the divergence theorem and \eqref{eq:upper bound for laplace wk},
	\begin{align*}
	&\ \ \ \tau_k(2b_k)-\tau_k(b_k/2)-\tau_k(b_k/4;z_k)= \frac{1}{\Omega_{n-1}}\int_{\Sigma_k^1\cap A(y_k,b_k/2,2b_k;M_k)\setminus B(z_k,b_k/4;M_k)} \Delta w_k\,dx\\   
	& \leq \frac{1}{\Omega_{n-1}}\int_{\Sigma_k^1\cap A(y_k,b_k/2,2b_k;M_k)\setminus B(z_k,b_k/4;M_k)}\Big(w_k+\frac{w_k^3}{8\rho^4}\Big)\,dx \\ 
	&<  \frac{10}{3}r_k\cdot (1+\frac{1}{100})(2b_k)^n+(1+\frac{1}{100})(\frac{10}{3})^3\cdot \frac{1}{8}r_k^3(2b_k)^{n-4}<4r_k(2b_k)^n+r_k^2(2b_k)^{n-3}.
	\end{align*}
	Here 
	\[\tau_k(s;z_k)=\frac{1}{\Omega_{n-1}}\int_{\Sigma_k^1\cap\partial B(z_k,s,M_k)}\langle \nabla w_k, \bm \eta_z\rangle , \]
	and $\bm \eta_z $ is the unit normal to $\partial B(z_k,s;M_k)\cap \Sigma_k^1$. The first inequality is from \eqref{eq:upper bound for laplace wk}; we used \eqref{eq:bound wk on annuli:high} in the second one, and the area upper bound is from the fact that $b_k^{-1}(\Sigma_k^1-y_k)$ is very close to a hyperplane. Recall that $\Sigma_k$ is very close to a catenoid with radius $\wti r_k\leq r_k$ around $z_k$. Then by the same argument as in Proposition \ref{prop:bound wk:high}, we also have 
	\[ \tau_k(b_k/4;z_k)< 3\wti r_k^{n-1}+ 9\Big(\frac{b_k}{4}\Big)^n\wti r_k+\Big(\frac{b_k}{4}\Big)^{n-\frac{5}{2}}\wti r_k^{\frac{3}{2}}\leq 3r_k^{n-1}+ 9\Big(\frac{b_k}{4}\Big)^nr_k+\Big(\frac{b_k}{4}\Big)^{n-\frac{5}{2}}r_k^{\frac{3}{2}}.  \]
	It follows that 
	\begin{align*}\label{eq:tau 2bk:high}
	\tau_k(2b_k)&< \tau_k(b_k/2)+\tau_k(b_k/4;z_k)+4r_k(2b_k)^n+r_k^2(2b_k)^{n-3}\\
	&<  6r_k^{n-1}+ 9\Big(\frac{b_k}{2}\Big)^nr_k+\Big(\frac{b_k}{2}\Big)^{n-\frac{5}{2}}r_k^{\frac{3}{2}}+9\Big(\frac{b_k}{4}\Big)^nr_k+\Big(\frac{b_k}{4}\Big)^{n-\frac{5}{2}}r_k^{\frac{3}{2}}+4r_k(2b_k)^n+r_k^2(2b_k)^{n-3}\nonumber\\
	&< 6r_k^{n-1}+9(2b_k)^nr_k+ (2b_k)^{n-\frac{5}{2}}r_k^{\frac{3}{2}},\nonumber
	\end{align*}
	which implies that 
	\[ \wti\tau_k(2b_k) =\tau_k(2b_k)-9(2b_k)^nr_k-r_k^{\frac{3}{2}}(2b_k)^{n-\frac{5}{2}}< 6r_k^{n-1}.    \]
	Then one can prove that $\wti{\mc I}_k(s)$ and $\wti\tau_k(s)$ are decreasing on $[2b_k,\epsilon]$ by the same argument as in Proposition \ref{prop:bound wk:high}. 
	
	Applying Proposition \ref{prop:2nd order estimates}\eqref{item:harnack} and \eqref{eq:gammas round} again,
	\[ \max_{x\in \Sigma_k^1\cap \partial B(y_k,\epsilon;M_k) }w_k\leq (1+1/10)\mc I_k(\epsilon) \leq \frac{9}{2}r_k.  \]
	This completes the proof of Proposition \ref{prop:bound wk:high}.
\end{proof}

Then using the same argument as Theorem \ref{thm:hausdorff distance}, one can prove the Hausdorff distance upper bounds between $\Sigma_k$ and $S_0^n$. The only modification is to replace Propositions \ref{prop:bound wk} and \ref{prop:bound wk both} by Propositions \ref{prop:bound wk:high} and \ref{prop:Ik bound pass zk}.
\begin{theorem}\label{thm:hausdorff distance:high}
	For all sufficiently large $k$,
	\[  \max_{\Sigma_k}\dist_{M_k}(x,S^n_0)< 5r_k.\]
\end{theorem}

\appendix

\section{Minimal graphs}\label{sec:graph functions}
Let $\mc N\subset M$ be a two-sided minimal hypersurface in $(M^{n+1},g)$ with a chosen unit normal vector field $\n$. Let $d$ be the oriented distance function to $\mc N$, and $\mc N_s$ be the level set of $d$. Then for some small $\mk d>0$,  $\{\mc N_{t}\}_{t\in(-\mk d, \mk d)}$ forms a foliation of a neighborhood of $\mc N$. Denote by $\wti \nabla$ and $\nabla $ the Levi-Civita connections of $M$ and $\mc N_s$ respectively. Then $\wti \nabla d$ is the unit normal vector field on $\mc N_s$. 
In this section, we always assume that $M$ and $\mc N$ satisfy the following conditions: 
\begin{equation}\label{eq:appendix assumptions}
|A|_{\mc N}<\epsilon/{\mk d},\quad |R|\leq 1, \quad |\wti\nabla R|<C(n). 
\end{equation}
Here $R$ is the Riemannian curvature tensor of $M$; $\epsilon$ is a small constant depending only on $n$, e.g. $\epsilon=10^{-1000n}$; $C(n)>1$ is a constant that can be changed from line to line.

Let $\Sigma$ be a minimal graph over $\mc N$ with
\[ \max_{x\in \Sigma}\dist_{M}(x,\mc N)< \mk d.\]
Denote by $u$ the graph function. Then such a function can be extended to a neighborhood of $\mc N$ by taking 
\[ u(p, s)= u(p)-s.  \]
When restricted to $\mc N_s$, $u$ is the graph function of $\Sigma$ over $\mc N_s$. Moreover, $u = 0$ when restricted to $\Sigma$, and hence $\wti \nabla u|_{\Sigma}$ is the normal vector field on $\Sigma$ and is non-zero everywhere.

For $p\in\Sigma\cap \mc N_s$, let $\{e_i\}$ be an orthonormal basis of $T_p\mc N_s$. Let $A$ and $H$ denote respectively the second fundamental form and mean curvature of $\mc N_s$ with respect to $\wti\nabla d$. A direct computation gives that 
\begin{gather*}
 \frac{\partial }{\partial s}H= -\Ric(\wti \nabla d,\wti \nabla d)-|A|^2;\\
\wti \nabla_{\wti \nabla d}\wti \nabla u=\wti \nabla_{\wti \nabla d} \nabla u=-A(\nabla u); \ \ \ \frac{\partial }{\partial s}|\wti\nabla u|^2=-2A(\nabla u,\nabla u);\\
 \dv_MA(\nabla u)=\dv_{\mc N_s}A(\nabla u)=\langle \nabla^2 u,A\rangle +\Ric(\nabla u,\wti \nabla d)+\langle \nabla u,\nabla H\rangle;\\
 \frac{\partial}{\partial s}\dv_M\wti \nabla u =-\Ric(\wti \nabla d,\wti \nabla u+\nabla u)+|A|^2-2\langle \nabla^2 u,A\rangle-\langle \nabla u,\nabla H\rangle;\\
 \frac{\partial}{\partial s}|\nabla H|^2=-2\langle\nabla (\Ric (\wti \nabla d,\wti \nabla d)+|A|^2),\nabla H\rangle -2A(\nabla H,\nabla H);\\
 \frac{\partial }{\partial s}|A|^2=2R(\wti \nabla d,e_i,\wti\nabla d,e_j)A(e_i,e_j)-2A(e_i,e_j)A(e_j,e_k)A(e_k,e_i);\\
 \frac{\partial }{\partial s}|\nabla A|^2=2R_{sjsk,i}A_{jk,i}+ 4R_{silk}A_{jl}A_{ki,j}+4R_{silj}A_{kl}A_{jk,i}-2A_{jk,i}A_{jk,l}A_{li}-4A_{ik,j}A_{lk,j}A_{il};\\
 \frac{\partial }{\partial s}|\nabla^2 u|^2=-2A_{kj,i}u_ku_{ij}+2R_{sikj}u_ku_{ij}-4u_{ij}u_{kj}A_{ik}.
\end{gather*}

We pause to state a standard differential inequality, whose proof is left to readers.
\begin{lemma}\label{lem:dif ineq}
Let $f:[0, \mk d]\to\mb R$ be a non-negative differentiable function. Suppose that 
\[ f'(t)\leq a+bf(t) \]
for real numbers $a$ and $b$. Then for each $t\in[0, \mk d]$,
\[f(t)\leq e^{bt}f(0)+\frac{a}{b}(e^{bt}-1).\]	
	\end{lemma}

 Now Lemma \ref{lem:dif ineq} can be applied to bound those terms on $\mc N_s$ by their restriction on $\mc N$.
 
\begin{lemma}
	\begin{gather*}
	   |\nabla u(x,t)|\le 2|\nabla u(x)|;\ \ \ \	|A(x,t)|\leq 2|A(x)|+2t;\\
	 \Big||A(x,t)|^2-|A(x)|^2\Big|\leq C(n)\Big(|A(x)|t+t+|A(x)|^3t\Big);\\
	  |\nabla A(x,t)|\leq C(n)\Big(|\nabla A(x)|+t(|A(x)|+1)\Big);\\
	    |\nabla H(x,t)|\leq C(n)\Big(t+t(|A(x)|+t)|\nabla A(x)|+|A(x)|^2t^2\Big);\\
	     |\nabla ^2u(x,t)|\leq C(n)\Big(|\nabla^2u(x)|+t(1+|\nabla A(x)|)|\nabla u(x)|\Big).
	\end{gather*}
	\end{lemma}
\begin{proof}
Note that $\frac{\partial }{\partial s}|A|^2\leq 2|A|+2|A|^3$. Then by classical OD-inequalities, we have
\[ \arctan |A(x,t)|\leq \arctan |A(x)|+t. \]	
Since $|tA(x)|\leq \delta|A(x)|<\epsilon$,  
\begin{equation}\label{eq:bound |A|}
  |A|(x,t)\leq 2|A(x)|+2t. 
  \end{equation}
The others can be derived similarly as follows: note that 
\[ \frac{\partial }{\partial s}|\nabla A|^2\leq C(n)\Big(|\nabla A(x,t)|+|A(x,t)||\nabla A(x,t)|+ |A(x,t)||\nabla A(x,t)|^2\Big).  \]
Fix $x\in\mc N$. Now let $f(t)=\sqrt{1+|\nabla A(x,t)|^2}$. It follows that 
\[  f'(t)\leq C(n)\Big(1+|A(x,t)|+|A(x,t)|f\Big)\leq C(n)\Big(1+|A(x)|+(1+|A(x)|)f(t)\Big) \]
Here the second inequality comes from \eqref{eq:bound |A|}. Then Lemma \eqref{lem:dif ineq} can be applied to obtain
\[  |\nabla A(x,t)|+1\leq e^{C(n)(1+|A(x)|)t}(1+|\nabla A(x)|)+e^{C(n)(1+|A(x)|)t}-1,   \]
which yields
\[ |\nabla A(x,t)|\leq 2|\nabla A(x)|+C(n)(1+|A(x)|)t.    \]	
Here we used the condition that $C(n)|A(x)|t\leq C(n)|A(x)|\mk d\ll 1$ by \eqref{eq:appendix assumptions}.
	\end{proof}

 Now we are ready to derive our inequality for $u$. Indeed, by the minimality of $\Sigma$,
\begin{align*}
0= \dv_{\Sigma}\wti \nabla u&=\dv_{M}\wti \nabla u-\wti \nabla^2u\Big(\frac{\wti\nabla u }{|\wti\nabla u|},\frac{\wti\nabla u }{|\wti\nabla u|}\Big)\\
&=\dv_{M}\wti \nabla u-  (A+ \nabla^2u)\Big(\frac{\nabla u }{|\wti \nabla u|},\frac{\nabla u }{|\wti\nabla u|}\Big).
\end{align*}
It follows that 
\begin{align*}
   \frac{(A+\nabla ^2u)(\nabla u,\nabla u)}{1+|\nabla u|^2}&=\dv_{\mc N}\nabla u+\int_0^s\frac{\partial }{\partial t}\dv_{M}\wti \nabla u \\
   &=\Delta_{\mc N} u+\int_0^s-\Ric(\wti \nabla d,\wti \nabla u+\nabla u)+|A|^2-2\langle \nabla^2 u,A\rangle-\langle \nabla u,\nabla H\rangle\,dt.
      \end{align*}
Since $\wti \nabla u=\nabla u-\wti\nabla d$, then
\[  \Ric(\wti \nabla d,\wti\nabla u+\nabla u)=-\Ric(\wti\nabla d, \wti \nabla d)+2\Ric(\wti\nabla d,\nabla u). \]                 
It follows that 
\begin{align*}     
  (\Delta u&+|A|^2u)(x,0)+\int_0^s\Ric(\wti\nabla d,\wti\nabla d)\,dt\\
  &\leq (|A|+|\nabla^2u|)|\nabla u|^2(x,s)+\int_0^s2n|\nabla u|+\Big||A(x)|^2-|A(x,t)|^2\Big|+2|\nabla^2 u||A| +|\nabla u||\nabla H|\,dt. 
  \end{align*}
Combining all of them, we conclude that
\begin{equation}\label{eq:minimal graph equation}
\begin{aligned} 
\Delta_{\mc N} u+|A|_{\mc N}^2u+&\int_0^{u(x)}\Ric(\wti\nabla d,\wti\nabla d)\,dt\\
&\leq 8|A|_{\mc N}|\nabla u|^2+|A|_{\mc N}^3u^2+|\nabla^2u||A|_{\mc N}u+C(n)\Big(u+|\nabla u|+|\nabla u|^3\Big)u+\\
&+C(n)\Big\{|\nabla u(x)|^2|\nabla^2 u|+|\nabla u|^3|\nabla A|u+|A|u^2(1+|\nabla u|)+|\nabla^2u|u^2+\\
&+|\nabla A||\nabla u|u^2(|A|+u)+|\nabla u|u^2(1+|A||\nabla u|+|\nabla A|u+|A|^2u)\Big\}_{\mc N}.
\end{aligned}
\end{equation}

\section{Catenoids}\label{sec:catenoids}
In this section, we collect some basic results for the catenoids. In $\mb R^{n+1}$, given $r>0$, there is an associated catenoid given by
\[ |x_{n+1}|=\int_r^t\frac{ds}{\sqrt{(s/r)^{2(n-1)}-1} }, \ \ t=\sqrt{x_1^2+x_2^2+\cdots x_n^2}. \]
Here $r$ is called the {\em radius}. When $r=1$, this catenoid is said to be {\em standard}. Let $t,h,R$ be the solutions of 
\begin{equation}\label{eq:catenoid equation}
	  t^2+h^2=R^2 \ \ \ \text { and }\ \ \ |h|=\int_1^t\frac{ds}{\sqrt{s^{2(n-1)}-1} }.  
	  \end{equation}
For a standard catenoid $\mc C$, we have that
\begin{enumerate} 
	\item\label{item:h bound} for $n=2$, $\log t<|h|<\log(2t)$ and for $n\geq 3$, $|h|<1.31103$;
	\item $|A(x)|\leq \sqrt{n(n-1)}$ and $|x||A|\to0$ as $|x|\to\infty$;
	\item for each connected component $\gamma$ of $\partial B_R(0)\cap \mc C$, it bounds an $n$-dimensional ball $D$ with area $\frac{\Omega_{n-1}}{n}t^n$, where $\Omega_{n-1}$ is the area of unit $(n-1)$-sphere;
	\item for $4\leq (n+1)\leq 7$,
 \begin{equation}\label{def:An}
   \mc A_n:=\lim_{R\to\infty} \Area (\mc C\cap B_R(0))-2\cdot \frac{\Omega_{n-1}}{n}t^n>0;
   \end{equation}
 \item for $n=2$,
 	\begin{equation}\label{eq:area difference}
 \Area(\mc C\cap B_R(0))-2\Area(D)=2\pi(|h|+t\sqrt{t^2-1})-2\pi t^2> 2\pi(\log R-1).   
 \end{equation}  
	\end{enumerate}
Note that outside $B_2(0)$, $\mc C$ has two connected components $\Sigma^1$ and $\Sigma^2$. Here we assume $\Sigma^1\subset \{x_{n+1}<0\}$. Then $\Sigma^2$ is a graph over $\Sigma^1$. Let $w$ denote the graph function. Denote by $\bm \eta$ the unit outward normal to $\partial B_R(0)\cap \Sigma^1$ of $B_R(0)\cap \Sigma^1$.
\begin{proposition}\label{prop:standard nabla w:high dim}
	\[  \lim_{R\to\infty} \frac{1}{\Omega_{n-1}}\int_{\Sigma_k^1\cap\partial B_R(0)}\langle \nabla w,\bm \eta\rangle =2.  \]
\end{proposition}
\begin{proof}[Proof of Proposition \ref{prop:standard nabla w:high dim}]
By computation, 
\[  \partial_h=(-\sqrt {t^{2n-2}-1} Y,1), \]
where $Y=(y_1,y_2,\cdots,y_n)$ with $|Y|=1$. Then the normal line at $(tY,h)$ intersects $\mc C$ at another point
\[ P=(t_hY,z_h),  \]
where $z_h>0$ is given by 
\begin{equation}\label{eq:normal line}
   z_h=h+(t_h-t)\sqrt {t^{2n-2}-1}  . 
\end{equation}  
   
	\begin{figure}[h]
	\begin{center}
		\def\svgwidth{0.7\columnwidth}
		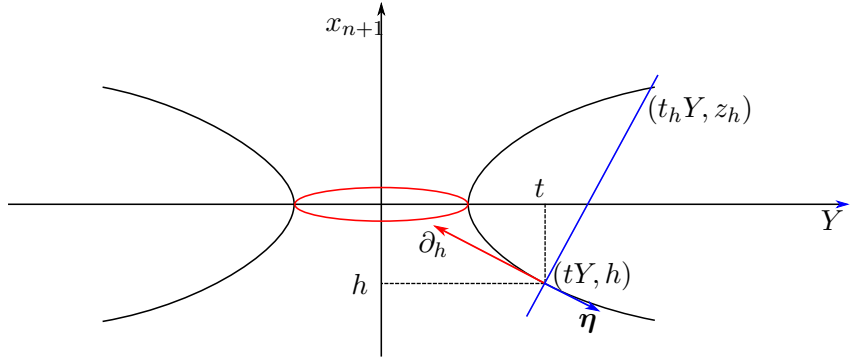
		\caption{The normal line of a catenoid.}
		\label{fig:Catenoid}
	\end{center}
        \end{figure}

Observe that $z_h>0$ and $h<0$. Then by \eqref{eq:normal line}, $t_h>t$. By  \eqref{eq:catenoid equation}, $z_h<2\log t_h<t_h$ and $|h|<2\log t<t$ for large $t$. Plugging them into \eqref{eq:normal line}, it follows that 
\[  t_h+t>z_h-h=(t_h-h)\sqrt{t^{2n-2}-1}\geq 3(t_h-t), \]
which implies that $t_h<2t$. Plugging it back into \eqref{eq:normal line} again,
\[  2\log (2t) +2\log t>z_h-h=(t_h-t)\sqrt{t^{2n-2}-1}> (t_h-t)(t-1). \]
It follows that  $t<t_h<t+1$ for large $t$. Observe that by \eqref{eq:catenoid equation},
\[ \partial_h t_h =\partial_hz_h\cdot \sqrt{t_h^{2n-2}-1}; \quad \partial_h t=-\sqrt{t^{2n-2}-1}.     \]
Differentiating both sides in \eqref{eq:normal line}, we then have
\[ \partial_ht_h=1+(\partial_ht_h+\partial_h t)\sqrt{t^{2n-2}-1}+(t_h-t)(n-1)\frac{t^{2n-3}\cdot\partial_ht}{\sqrt{t^{2n-2}-1}}.     \]
Combining them together, it follows that
\[  \partial_hz_h=\frac{t^{2n-2}-(t_h-t)(n-1)t^{2n-3}}{1-\sqrt {(t_h^{2n-2}-1)(t^{2n-2}-1)}}\to -1 \ \ \ \text{ as }\ \ \ h\to-\infty. \]
Note that 
\[ w^2=( t_h-t )^2+(z_h-h)^2. \]
Then by computation,
\begin{align*}
  \partial_hw^2&=2(t_h-t)(\sqrt{t_h^{2n-2}-1}\partial_hz_h+\sqrt{t^{2n-2}-1})+2(z_h-h)(\partial _hz_h-1)   \\
  &=  2\cdot \Big(\frac{\sqrt{t_h^{2n-2}-1}}{\sqrt{t^{2n-2}-1}}+1\Big)\cdot(z_h-h) \partial_h z_h,
  \end{align*}
  which implies that 
  \[  \lim_{h\to-\infty}\partial_h w=\lim_{h\to-\infty} -\frac{2(z_h-h)}{w} =\lim_{h\to-\infty} -2\cdot \frac{(t_h-t)\sqrt{t^{2n-2}-1}} {t^{n-1}(t_h-t)}=-2.  \]
  On the other hand,
  \[\frac{1}{\Omega_{n-1}}\int_{\Sigma\cap\{ x_{n+1}=h \}}\langle \nabla w,\bm \eta\rangle =\frac{1}{\Omega_{n-1}} \int_{\Sigma\cap\{ z=h \}} \frac{-\partial_h w}{t^{n-1}} =-\partial_h w.
\]
  This finishes the proof of Proposition \ref{prop:standard nabla w:high dim}.
	\end{proof}

\bibliographystyle{amsalpha}
\bibliography{minmax}
\end{document}

%% file: Structure_of_Sigmak.eps_tex
\begingroup%
  \makeatletter%
  \providecommand\color[2][]{%
    \errmessage{(Inkscape) Color is used for the text in Inkscape, but the package 'color.sty' is not loaded}%
    \renewcommand\color[2][]{}%
  }%
  \providecommand\transparent[1]{%
    \errmessage{(Inkscape) Transparency is used (non-zero) for the text in Inkscape, but the package 'transparent.sty' is not loaded}%
    \renewcommand\transparent[1]{}%
  }%
  \providecommand\rotatebox[2]{#2}%
  \newcommand*\fsize{\dimexpr\f@size pt\relax}%
  \newcommand*\lineheight[1]{\fontsize{\fsize}{#1\fsize}\selectfont}%
  \ifx\svgwidth\undefined%
    \setlength{\unitlength}{470.194691bp}%
    \ifx\svgscale\undefined%
      \relax%
    \else%
      \setlength{\unitlength}{\unitlength * \real{\svgscale}}%
    \fi%
  \else%
    \setlength{\unitlength}{\svgwidth}%
  \fi%
  \global\let\svgwidth\undefined%
  \global\let\svgscale\undefined%
  \makeatother%
  \begin{picture}(1,0.27928635)%
    \lineheight{1}%
    \setlength\tabcolsep{0pt}%
    \put(0,0){\includegraphics[width=\unitlength]{Structure_of_Sigmak.eps}}%
    \put(0.35352698,0.13826155){\color[rgb]{0,0,0}\makebox(0,0)[lt]{\lineheight{1.25}\smash{\begin{tabular}[t]{l}link of the catenoid around $y_{k,1}$\end{tabular}}}}%
    \put(0.57540752,0.2054209){\color[rgb]{0,0,0}\makebox(0,0)[lt]{\lineheight{1.25}\smash{\begin{tabular}[t]{l}link of the catenoid around $y_{k,2}$\end{tabular}}}}%
  \end{picture}%
\endgroup%

%% file: Replacing_catenoids_by_two_disks.eps_tex
\begingroup%
  \makeatletter%
  \providecommand\color[2][]{%
    \errmessage{(Inkscape) Color is used for the text in Inkscape, but the package 'color.sty' is not loaded}%
    \renewcommand\color[2][]{}%
  }%
  \providecommand\transparent[1]{%
    \errmessage{(Inkscape) Transparency is used (non-zero) for the text in Inkscape, but the package 'transparent.sty' is not loaded}%
    \renewcommand\transparent[1]{}%
  }%
  \providecommand\rotatebox[2]{#2}%
  \newcommand*\fsize{\dimexpr\f@size pt\relax}%
  \newcommand*\lineheight[1]{\fontsize{\fsize}{#1\fsize}\selectfont}%
  \ifx\svgwidth\undefined%
    \setlength{\unitlength}{447.87266929bp}%
    \ifx\svgscale\undefined%
      \relax%
    \else%
      \setlength{\unitlength}{\unitlength * \real{\svgscale}}%
    \fi%
  \else%
    \setlength{\unitlength}{\svgwidth}%
  \fi%
  \global\let\svgwidth\undefined%
  \global\let\svgscale\undefined%
  \makeatother%
  \begin{picture}(1,0.50519223)%
    \lineheight{1}%
    \setlength\tabcolsep{0pt}%
    \put(0,0){\includegraphics[width=\unitlength]{Replacing_catenoids_by_two_disks.eps}}%
    \put(0.08352695,0.27940464){\color[rgb]{0,0,0}\makebox(0,0)[lt]{\lineheight{1.25}\smash{\begin{tabular}[t]{l}$\Sigma_k^2$\end{tabular}}}}%
    \put(0.07235359,0.22905447){\color[rgb]{0,0,0}\makebox(0,0)[lt]{\lineheight{1.25}\smash{\begin{tabular}[t]{l}\\$\Sigma_k^1$\end{tabular}}}}%
    \put(0.50650327,0.2851127){\color[rgb]{0,0,0}\makebox(0,0)[lt]{\lineheight{1.25}\smash{\begin{tabular}[t]{l}$\mathcal D_k^2$\end{tabular}}}}%
    \put(0.51546407,0.22012273){\color[rgb]{0,0,0}\makebox(0,0)[lt]{\lineheight{1.25}\smash{\begin{tabular}[t]{l}$\mathcal D_k^1$\end{tabular}}}}%
    \put(0.49498771,0.175391){\color[rgb]{0,0,0}\makebox(0,0)[lt]{\lineheight{1.25}\smash{\begin{tabular}[t]{l}$B(y_k,R_k;M_k)$\end{tabular}}}}%
    \put(0.45005556,0.02955663){\color[rgb]{0,0,0}\makebox(0,0)[lt]{\lineheight{1.25}\smash{\begin{tabular}[t]{l}$B(y_k,\epsilon;M_k)$\end{tabular}}}}%
    \put(0.82731643,0.1848111){\color[rgb]{0,0,0}\makebox(0,0)[lt]{\lineheight{1.25}\smash{\begin{tabular}[t]{l}$S_0^n$\end{tabular}}}}%
  \end{picture}%
\endgroup%

%% file: Catenoid.eps_tex
\begingroup%
  \makeatletter%
  \providecommand\color[2][]{%
    \errmessage{(Inkscape) Color is used for the text in Inkscape, but the package 'color.sty' is not loaded}%
    \renewcommand\color[2][]{}%
  }%
  \providecommand\transparent[1]{%
    \errmessage{(Inkscape) Transparency is used (non-zero) for the text in Inkscape, but the package 'transparent.sty' is not loaded}%
    \renewcommand\transparent[1]{}%
  }%
  \providecommand\rotatebox[2]{#2}%
  \newcommand*\fsize{\dimexpr\f@size pt\relax}%
  \newcommand*\lineheight[1]{\fontsize{\fsize}{#1\fsize}\selectfont}%
  \ifx\svgwidth\undefined%
    \setlength{\unitlength}{406.83466372bp}%
    \ifx\svgscale\undefined%
      \relax%
    \else%
      \setlength{\unitlength}{\unitlength * \real{\svgscale}}%
    \fi%
  \else%
    \setlength{\unitlength}{\svgwidth}%
  \fi%
  \global\let\svgwidth\undefined%
  \global\let\svgscale\undefined%
  \makeatother%
  \begin{picture}(1,0.42169238)%
    \lineheight{1}%
    \setlength\tabcolsep{0pt}%
    \put(0,0){\includegraphics[width=\unitlength]{Catenoid.eps}}%
    \put(0.37693292,0.38853423){\color[rgb]{0,0,0}\makebox(0,0)[lt]{\lineheight{1.25}\smash{\begin{tabular}[t]{l}$x_{n+1}$\end{tabular}}}}%
    \put(0.40813262,0.07514315){\color[rgb]{0,0,0}\makebox(0,0)[lt]{\lineheight{1.25}\smash{\begin{tabular}[t]{l}$h$\end{tabular}}}}%
    \put(0.6252071,0.19089476){\color[rgb]{0,0,0}\makebox(0,0)[lt]{\lineheight{1.25}\smash{\begin{tabular}[t]{l}$t$\end{tabular}}}}%
    \put(0.48733904,0.12525524){\color[rgb]{0,0,0}\makebox(0,0)[lt]{\lineheight{1.25}\smash{\begin{tabular}[t]{l}$\partial_h$\end{tabular}}}}%
    \put(0.67749532,0.03658939){\color[rgb]{0,0,0}\makebox(0,0)[lt]{\lineheight{1.25}\smash{\begin{tabular}[t]{l}$\bm\eta$\end{tabular}}}}%
    \put(0.64616126,0.08802953){\color[rgb]{0,0,0}\makebox(0,0)[lt]{\lineheight{1.25}\smash{\begin{tabular}[t]{l}$(tY,h)$\end{tabular}}}}%
    \put(0.96542715,0.15125672){\color[rgb]{0,0,0}\makebox(0,0)[lt]{\lineheight{1.25}\smash{\begin{tabular}[t]{l}$Y$\end{tabular}}}}%
    \put(0.75856348,0.28599394){\color[rgb]{0,0,0}\makebox(0,0)[lt]{\lineheight{1.25}\smash{\begin{tabular}[t]{l}$(t_hY,z_h)$\end{tabular}}}}%
  \end{picture}%
\endgroup%